\documentclass[11pt]{article}
\usepackage{authblk}
\usepackage[toc,page]{appendix}
\usepackage[top=2cm, bottom=2cm, left=2cm, right=2cm]{geometry}

\usepackage{color}
\usepackage{helvet}         
\usepackage{courier}        
\usepackage{type1cm}        

\usepackage{framed}
\usepackage{tikz}
\usepackage{makeidx}         
\usepackage{graphicx}        
\usepackage{multicol}        
\usepackage[bottom]{footmisc}
\usepackage{amsmath}
\usepackage{amssymb}
\usepackage{bbold}
\usepackage{amsthm}
\usepackage{subcaption}
\usepackage{sidecap}
\usepackage{floatrow}
\usepackage{pdflscape}
\usepackage{comment}
\usepackage[font=small]{caption}
\usepackage{enumitem}
\usepackage{esint}

\usepackage{scalerel}

\newtheorem{theorem}{Theorem}

\newtheorem{lemma}[theorem]{Lemma}

\newtheorem{proposition}[theorem]{Proposition}

\theoremstyle{definition}
\newtheorem{definition}[theorem]{Definition}
\theoremstyle{remark}

\numberwithin{theorem}{section}
\numberwithin{figure}{section}
\numberwithin{equation}{section}

\begin{document}

\title{Multiple SLEs and Dyson Brownian motion: \\transition density and Green's function}
\bigskip{}
\author[1]{Chongzhi Huang\thanks{huangchzh2001prob@gmail.com}}
\author[2]{Hao Wu\thanks{hao.wu.proba@gmail.com}}
\author[2]{Lu Yang\thanks{luyang\_proba@126.com}}
\affil[1]{Qiuzhen College, Tsinghua University}
\affil[2]{Yau Mathematical Sciences Center, Tsinghua University}

\date{}

%
%

\maketitle
\vspace{-1cm}
\begin{center}
\begin{minipage}{0.95\textwidth}
\abstract We consider multiple chordal Schramm-Loewner evolution (SLE) with $\kappa\in (0,4]$. Under common-time parameterization, we show that the transition density of the driving function of multiple chordal SLEs can be given by the transition density of Dyson Brownian motion and Green's function. 
This is a companion paper of~[FWY24]. 

\end{minipage}
\end{center}
\tableofcontents

\newpage
\section{Introduction}
In this article, we consider connections between multiple Schramm-Loewner evolution (SLE) and Dyson Brownian motion. 
The motivation to study multiple SLEs is to describe the scaling limit of multiple interfaces in polygons, see for instance~\cite{BauerBernardKytolaMultipleSLE, DubedatCommutationSLE, KytolaPeltolaPurePartitionFunctions}.
The study on its connections with Dyson Brownian motion goes back to~\cite{CardySLEDysonBM} in multiple radial case, and some recent research including~\cite{KatoriKoshidaThreePhasesofMultipleSLSE, ChenMargarintPerturbationsofMultipleSLE, CampbellLuhMargarintCvgMultipleSLE} has also proved properties of multiple SLEs driven by Dyson Brownian motion, which is a local formulation of multiple SLEs defined through multiple Loewner equations. 

There is another global formulation of multiple chordal SLEs defined as commuting Loewner chains that we describe below. 
We say that $(\Omega; x_1, \ldots, x_p)$ is a (topological) $p$-polygon if  $\Omega\subsetneq\mathbb{C}$ is simply connected and $x_1, x_2, \ldots, x_p \in \partial\Omega$ are distinct points lying counterclockwise along the boundary. 
We will also assume throughout that $\partial\Omega$ is locally connected and the marked boundary points $x_1, x_2, \ldots, x_p$ lie on $C^{1+\epsilon}$-boundary segments, for some $\epsilon>0$.
We frequently use the reference polygon $(\mathbb{H}; \boldsymbol{x})$ which is the upper-half plane $\mathbb{H}:=\{z\in\mathbb{C}: \Im{z}>0\}$ with marked boundary points $\boldsymbol{x}=(x_1, \ldots, x_p)\in\mathcal{X}_p$ where
\begin{align*}
\mathcal{X}_p := \{\boldsymbol{x} = (x_1, \ldots, x_p)\in\mathbb{R}^p : x_1<\cdots<x_p\}.
\end{align*}

In this article, we focus on polygons with even number of marked points, i.e. polygons $(\Omega; x_1, \ldots, x_{2n})$ with $p=2n$. We consider $n$ non-crossing simple curves in $\Omega$ such that each curve connects two points among $\{x_1, \ldots, x_{2n}\}$. These curves can have various planar connectivities. We describe such connectivities by link patterns $\alpha=\{\{a_1, b_1\}, \ldots, \{a_n, b_n\}\}$ where $\{a_1, b_1, \ldots, a_n, b_n\}=\{1, 2, \ldots, 2n\}$, and denote by $\mathrm{LP}_n$ the set of all link patterns. Note that $\#\mathrm{LP}_n$ is the Catalan number $C_n=\frac{1}{n+1}\binom{2n}{n}$. See Figure~\ref{fig::6points}. 

\begin{figure}[ht!]
\begin{subfigure}[b]{0.3\textwidth}
\begin{center}
\includegraphics[width=0.8\textwidth]{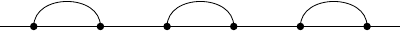}
\end{center}
\caption{$\{\{1,2\}, \{3,4\}, \{5,6\}\}$}
\end{subfigure}
\begin{subfigure}[b]{0.3\textwidth}
\begin{center}
\includegraphics[width=0.8\textwidth]{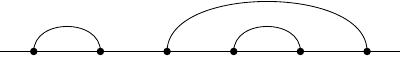}
\end{center}
\caption{$\{\{1,2\}, \{3,6\}, \{4,5\}\}$}
\end{subfigure}
\begin{subfigure}[b]{0.3\textwidth}
\begin{center}
\includegraphics[width=0.8\textwidth]{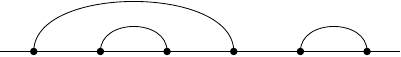}
\end{center}
\caption{$\{\{1,4\}, \{2,3\}, \{5,6\}\}$}
\end{subfigure}\\
\vspace{0.5cm}
\begin{subfigure}[b]{0.3\textwidth}
\begin{center}
\includegraphics[width=0.8\textwidth]{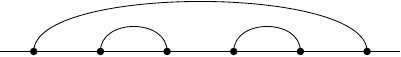}
\end{center}
\caption{$\{\{1,6\}, \{2,3\}, \{4,5\}\}$}
\end{subfigure}
\begin{subfigure}[b]{0.3\textwidth}
\begin{center}
\includegraphics[width=0.8\textwidth]{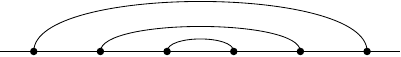}
\end{center}
\caption{$\{\{1,6\}, \{2,5\}, \{3,4\}\}$}
\end{subfigure}
\caption{\label{fig::6points} When $n=3$, there are $C_3=5$ link patterns. 
}
\end{figure}

When $n=1$, we denote by $\mathfrak{X}(\Omega; x_1, x_2)$ the set of continuous simple unparameterized curves in $\Omega$ connecting $x_1$ and $x_2$ such that they only touch the boundary $\partial\Omega$ in $\{x_1, x_2\}$. Chordal SLE in $(\Omega; x_1, x_2)$ is a random curve in $\mathfrak{X}(\Omega; x_1, x_2)$ when $\kappa\le 4$.  In general, for $n\ge 2$, we consider multiple SLEs defined via resampling property. Such resampling property is inherited from multiple interfaces for critical lattice model in a polygon. 

\begin{definition}[Multichordal SLE]\label{def::globalnSLE}
Fix $\kappa\in (0,4]$ and $n\ge 2$. Let $(\Omega; \boldsymbol{x})=(\Omega; x_1, \ldots, x_{2n})$ be a $2n$-polygon. Fix $\alpha=\{\{a_1, b_1\}, \ldots \{a_n, b_n\}\}\in\mathrm{LP}_n$. Consider curves $\gamma^j\in \mathfrak{X}(\Omega; x_{a_j}, x_{b_j})$ for each $j$; and suppose $\gamma^i\cap\gamma^j=\emptyset$ for $i\neq j$. We call the law $\mathbb{Q}_{\alpha}(\Omega; \boldsymbol{x})$ of $\boldsymbol{\gamma}=(\gamma^1, \ldots, \gamma^n)$ multichordal $\mathrm{SLE}_{\kappa}$ in $(\Omega; \boldsymbol{x})$ associated to link pattern $\alpha$ if it satisfies the resampling property: for each $j\in\{1, \ldots, n\}$, the conditional law of $\gamma^j$ given $\{\gamma^i: i\neq j\}$ is chordal $\mathrm{SLE}_{\kappa}$ from $x_{a_j}$ to $x_{b_j}$ in the connected component $\Omega_j$ of $\Omega\setminus\cup_{i\neq j}\gamma^i$ having $x_{a_j}, x_{b_j}$ on its boundary. In particular, when $(\Omega; \boldsymbol{x})=(\mathbb{H}; \boldsymbol{x})$,  we denote by $\mathbb{Q}_{\alpha}(\boldsymbol{x})$ the law of multichordal $\mathrm{SLE}_{\kappa}$ in $(\mathbb{H}; \boldsymbol{x})$ associated to link pattern $\alpha$. 
\end{definition}

The existence and uniqueness of multichordal $\mathrm{SLE}_{\kappa}$ for different ranges of $\kappa$ has been widely studied, see~\cite{MillerSheffieldIG1, MillerSheffieldIG2, PeltolaWuGlobalMultipleSLEs, WuHyperSLE, BeffaraPeltolaWuUniqueness, AngHoldenSunYu2023, ZhanExistenceUniquenessMultipleSLE, FengLiuPeltolaWu2024}.  
In particular, the existence and uniqueness holds for all $\kappa\in (0,8)$. In this article, we focus on $\kappa\le 4$. 
In the companion paper~\cite{FengWuYangIsing}, we derived the connection between multichordal SLE and Dyson circular ensemble for $\kappa\in (0,4]$ in radial coordinate. 
In this article, we will prove corresponding conclusions for multichordal SLE and Dyson Brownian motion for $\kappa\in (0,4]$ in chordal coordinate.  
The main conclusions are summarized as follows. 
\begin{itemize}
\item We introduce multi-time local martingale for multiple independent chordal SLEs, see Proposition~\ref{prop::multitime_mart}. Its analogue in radial coordinates is constructed in~\cite{HuangPeltolaWuMultiradialSLE}. These multi-time local martingales are powerful tools in analyzing multiple SLEs and they will be frequently used in future projects: we use them to derive large deviation principle for multiple SLEs in~\cite{ChenHuangWuLDP}. We use them to construct solutions to Dub\'edat's commutation relation in~\cite{HuangWuBPZRadial, ChenHuangWuBPZChordal}.
\item We parameterize multichordal SLEs by common-time parameter (see definition in Section~\ref{subsec::commontime_mart}) and derive the transition density of the driving function in Theorem~\ref{thm::transitiondensity}: it can be given by the transition density of Dyson Brownian motion and Green's function. 
Using such connection and the asymptotics of the transition density of Dyson Brownian motion, we derive the asymptotics of a rare event and show that, conditional on this rare event, the driving function of multichordal SLEs converges to Dyson Brownian motion (Proposition~\ref{prop::cvg_DysonBM}). 
\item We take $\kappa=4$ as an example and connect our result to Gaussian free field level lines (Proposition~\ref{prop::GFF_DysonBM}). In this case, the SDE for driving function~\eqref{eqn::GFFlevellines_SDE}, and Green's function~\eqref{eqn::Greenfunction_GFF} are all explicit, and the Dyson Brownian motion is of parameter $\beta=2$. 
Finally, we give one application of our main result: we derive the asymptotic of the probability for the area of certain geometric neighborhood of multichordal SLEs to be large (Theorem~\ref{thm::multichordalSLE_boundaryarm}). 
\end{itemize}

\subsection{Transition density and Green's function}
\label{subsec::intro_transition}
We focus on the $2n$-polygon $(\mathbb{H}; \boldsymbol{x})$ with $\boldsymbol{x}=(x_1, \ldots, x_{2n})\in\mathcal{X}_{2n}$. Recall that $\mathbb{Q}_{\alpha}(\boldsymbol{x})$ denotes the law of multichordal $\mathrm{SLE}_{\kappa}$ in  $(\mathbb{H}; \boldsymbol{x})$ associated to link pattern $\alpha=\{\{a_1, b_1\}, \ldots, \{a_n, b_n\}\}\in\mathrm{LP}_n$ as in Definition~\ref{def::globalnSLE}. Suppose $\boldsymbol{\gamma}=(\gamma^1, \ldots, \gamma^n)\sim\mathbb{Q}_{\alpha}(\boldsymbol{x})$ and suppose $\gamma^j$ is from $x_{a_j}$ to $x_{b_j}$ for $1\le j\le n$. We view $\boldsymbol{\gamma}$ as a $2n$-tuple of continuous simple curves: for $j\in\{1, \ldots, n\}$, we define $\eta^{a_j}$ to be $\gamma^j$ and $\eta^{b_j}$ to be the time-reversal of $\gamma^j$. In this way, $\eta^j$ is a continuous simple curve in $\mathbb{H}$ starting from $x_j$ for $j\in\{1, \ldots, 2n\}$. We will introduce common-time parameter in Section~\ref{subsec::commontime_mart}. Its analogue in radial coordinate was introduced in~\cite{HealeyLawlerNSidedRadialSLE}. 
We will prove in Lemma~\ref{lem::ppf_mart_common} that, under $\mathbb{Q}_{\alpha}(\boldsymbol{x})$, the driving function $\boldsymbol{X}_t$ of multichordal $\mathrm{SLE}_{\kappa}$ satisfies the system of SDEs: 
\begin{equation}\label{eqn::SDE_ppf}
	\mathrm{d} X_t^j=\sqrt{\kappa}\mathrm{d} B_t^j+\kappa(\partial_j\log\mathcal{Z}_\alpha)(\boldsymbol{X}_t)\mathrm{d} t+\sum_{\ell\ne j}\frac{2}{X_t^j-X_t^{\ell}}\mathrm{d} t, \qquad 1\leq j\leq 2n,
\end{equation}
where $\{B^j\}_{1\le j\le 2n}$ are independent standard Brownian motions and $\mathcal{Z}_{\alpha}$ is the pure partition function for multichordal SLE defined in Section~\ref{subsec::pre_NSLE}. The solution to~\eqref{eqn::SDE_ppf} exists until the lifetime $T=\inf\{t: X_t^j=X_t^k\text{ for some }j\neq k\}$. 
The main goal of this article is to compare the solution~\eqref{eqn::SDE_ppf} with Dyson Brownian motion.

\begin{theorem}\label{thm::transitiondensity}
Fix $\kappa\in (0,4], n\ge 1$ and link pattern $\alpha\in\mathrm{LP}_n$. 
\begin{itemize}
\item Denote by $\mathsf{Q}_{\alpha}(t; \cdot, \cdot)$ the transition density for the solution to~\eqref{eqn::SDE_ppf}.\footnote{By saying that $\mathsf{Q}_{\alpha}(t; \cdot, \cdot)$ is the transition density, we mean that for any $t>0$ and $\boldsymbol{x}\in\mathcal{X}_{2n}$, if $\boldsymbol{X}_0=\boldsymbol{x}$, then for any bounded measurable function $f$ on $\mathcal{X}_{2n}$, \[\mathbb{E}_{\alpha}(\boldsymbol{x})\left[\mathbb{1}_{\{T>t\}}f(\boldsymbol{X}_t)\right]=\int_{\mathcal{X}_{2n}}\mathsf{Q}_{\alpha}(t; \boldsymbol{x}, \boldsymbol{y})f(\boldsymbol{y})\mathrm{d} \boldsymbol{y}.\]}
\item Denote by $\mathsf{Q}_{2n}(t; \cdot, \cdot)$ the transition density for Dyson Brownian motion with parameter $\beta=8/\kappa$: 
\begin{equation}\label{eqn::SDE_halfwatermelon_intro}
\mathrm{d} X^j_t=\sqrt{\kappa}\mathrm{d} B_t^j+\sum_{k\neq j}\frac{4}{X_t^j-X_t^k}\mathrm{d} t, \qquad 1\le j\le 2n, 
\end{equation}
where $\{B^j\}_{1\le j\le 2n}$ are independent standard Brownian motions. 
\end{itemize}
Then we have 
\begin{equation}\label{eqn::transitiondensity_ppf}
\mathsf{Q}_{\alpha}(t; \boldsymbol{x}, \boldsymbol{y})=\mathsf{Q}_{2n}(t; \boldsymbol{x}, \boldsymbol{y})\frac{G_{\alpha}(\boldsymbol{x})}{G_{\alpha}(\boldsymbol{y})}, \qquad \text{for all }t\ge 0, \text{and } \boldsymbol{x}, \boldsymbol{y}\in\mathcal{X}_{2n},
\end{equation}
where $G_{\alpha}$ is the Green's function defined as 
\begin{equation}\label{eqn::Green}
G_{\alpha}(\boldsymbol{x}):=\frac{\mathcal{Z}_{2n}(\boldsymbol{x})}{\mathcal{Z}_{\alpha}(\boldsymbol{x})}, \qquad\text{with }\mathcal{Z}_{2n}(\boldsymbol x) := \prod_{1\leq i<j\leq 2n}(x_{j}-x_{i})^{\frac{2}{\kappa}},\qquad \text{for }\boldsymbol{x}=(x_1, \ldots, x_{2n})\in\mathcal{X}_{2n}, 
\end{equation}
 and $\mathcal{Z}_{\alpha}$ is the pure partition function for multichordal SLE defined in Section~\ref{subsec::pre_NSLE}. 
\end{theorem}

Dyson Brownian motion~\eqref{eqn::SDE_halfwatermelon_intro} is well-understood. In particular, its lifetime $T=\infty$, i.e. the solution to~\eqref{eqn::SDE_halfwatermelon_intro} exists for all time. Green's function~\eqref{eqn::Green} is a generalization of previous Green's function of a single SLE curve~\cite{LawlerMinkowskiSLERealLine} and two SLE curves~\cite{ZhanGreen2SLEboundary}, see more discussion in Section~\ref{sec::NSLE_estimate}. 
The proof of Theorem~\ref{thm::transitiondensity} relies on the construction of multi-time local martingales, see Section~\ref{subsec::multitime_mart}. 

The relation~\eqref{eqn::transitiondensity_ppf} provides a tool to analyze multichordal SLEs using Dyson Brownian motion. For instance, one may analyze multiple SLEs as $N\to \infty$ using~\eqref{eqn::transitiondensity_ppf} as the asymptotic of Dyson Brownian motion as $N\to\infty$ is known, see~\cite{dMS16, HK18, HS21}. One may analyze multichordal SLEs as $\kappa\to 0+$ using~\eqref{eqn::transitiondensity_ppf} as the asymptotic of Dyson Brownian motion as $\beta\to\infty$ is known, this is related to large deviation principle for multiple SLEs, see~\cite{PeltolaWangSLELDP, AHP24, ChenHuangWuLDP}. In the following, we analyze multichordal SLEs as $t\to \infty$ using~\eqref{eqn::transitiondensity_ppf} as the asymptotic of Dyson Brownian motion as $t\to\infty$ is known~\cite{RoslerGeneralizedHermiteDunklOperators} (see Lemma~\ref{lem::density*}).

\subsection{Convergence to Dyson Brownian motion}

\begin{proposition}\label{prop::cvg_DysonBM}
Assume the same notations as in Theorem~\ref{thm::transitiondensity}. 
Let $\boldsymbol{X}_t$ be the solution to~\eqref{eqn::SDE_ppf} and let $T=\inf\{t: X_t^j= X_t^k\text{ for some }j\neq k\}$ be its lifetime. 
We have
\begin{equation}\label{eqn::multichordal_lifetime}
\mathbb{Q}_{\alpha}(\boldsymbol{x})[T>t]=\mathcal{I}_{2n}^{-1}\mathcal{J}_{\alpha} G_{\alpha}(\boldsymbol{x}) \left(\frac{1}{\sqrt{\kappa t}}\right)^{A_{2n}^+}\left(1+O\left(\frac{|\boldsymbol{x}|}{\sqrt{t}}\right)\right),\qquad \text{as }t\to\infty,
\end{equation}
where $G_{\alpha}$ is Green's function defined in~\eqref{eqn::Green},
and $A_{2n}^+$ is the half-plane arm exponent for SLE \begin{equation}\label{eqn::SLE_halfplane_arm}
A_{2n}^+=\frac{n(4n+4-\kappa)}{\kappa}, 
\end{equation} 
and $\mathcal{I}_{2n}$ is the normalization constant defined in~\eqref{eqn::DysonBM_normalization} with $p=2n$, $\mathcal{J}_{\alpha}$ is the normalization constant defined in~\eqref{eqn::ppf_normalization}.  
Moreover,
\begin{itemize}
\item for $t>0$, denote by $\mathbb{Q}_{\alpha}^{(t)}(\boldsymbol{x})$ the law of the solution to~\eqref{eqn::SDE_ppf} started from $\boldsymbol{x}$ conditional on $\{T>t\}$; 
\item denote by $\mathbb{Q}_{2n}(\boldsymbol{x})$ the law of Dyson Brownian motion~\eqref{eqn::SDE_halfwatermelon_intro}  started from $\boldsymbol{x}$. 
\end{itemize}
Denote by $(\mathcal{F}_t, t\ge 0)$ the filtration generated by the Brownian motions in the SDEs~\eqref{eqn::SDE_ppf} or~\eqref{eqn::SDE_halfwatermelon_intro}. Then for any $s>0$, when both measures $\mathbb{Q}_{\alpha}^{(t)}(\boldsymbol{x})$ and $\mathbb{Q}_{2n}(\boldsymbol{x})$ are restricted to $\mathcal{F}_s$, the total variation distance between the two measures goes to zero as $t\to\infty$: 
\begin{equation}\label{eqn::cvg_tv}
\lim_{t\to\infty}\mathrm{dist}_{\mathrm{TV}}\left(\mathbb{Q}_{\alpha}^{(t)}(\boldsymbol{x}), \mathbb{Q}_{2n}(\boldsymbol{x})\right)=0. 
\end{equation}
\end{proposition}

\subsection{Estimate for multichordal SLEs}

We will relate the estimate~\eqref{eqn::multichordal_lifetime} to a geometric property of multichordal SLE. 
Suppose $\gamma\in \mathfrak{X}(\mathbb{H}; x, y)$. 
Denote by $B_r(z)$ the ball centered at $z\in\mathbb{C}$ with radius $r>0$. 
The size ``$\mathrm{hsiz}$" of $\gamma$ is introduced by~\cite{LLN09}: define $\mathrm{hsiz}(\gamma)$ to be the two-dimensional Lebesgue measure of the union of all balls centered by points in $\gamma$ that are tangent to the real line, i.e.
\begin{equation}\label{eqn::hsiz_def}
\mathrm{hsiz}(\gamma)=\mathrm{area}\left(\cup_{u+\mathfrak{i} v\in \gamma} B_v(u+\mathfrak{i} v)\right). 
\end{equation} 
See Figure~\ref{fig::hsiz}. The geometric quantity $\mathrm{hsiz}$ provides a natural Euclidean interpretation of the “size” of 
a curve. It is comparable to other geometric quantities, such as half-plane capacity, the area of the hyperbolic neighborhood of radius one, the area of all Whitney squares in the upper half plane that intersect the compact hull, and the area under the minimal Lipschitz function of norm $1$ that lies above the compact hull, see~\cite{LLN09, RW14}. 

\begin{figure}[ht!]
\begin{subfigure}[b]{0.45\textwidth}
\begin{center}
\includegraphics[width=\textwidth]{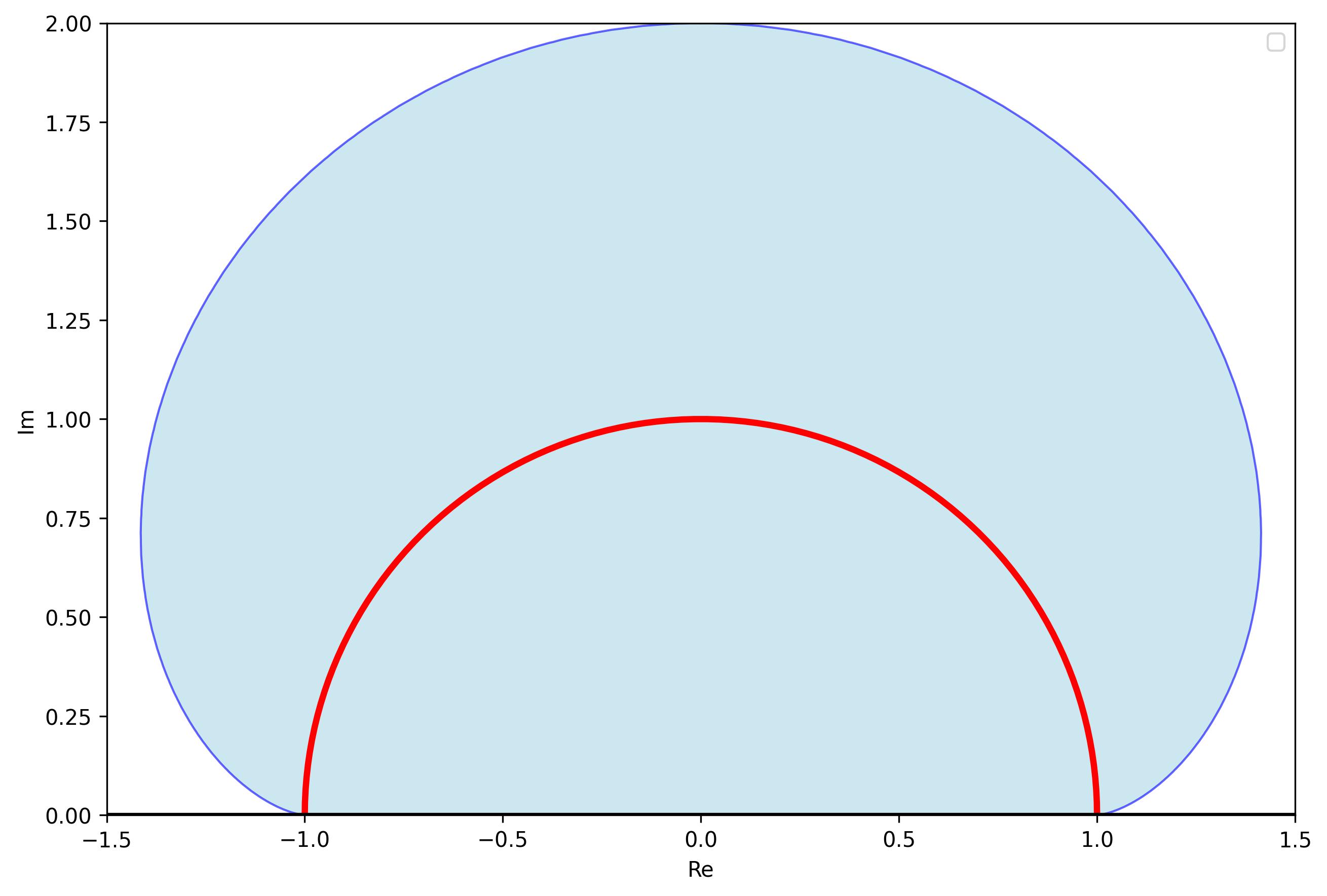}
\end{center}
\caption{}
\end{subfigure}
\begin{subfigure}[b]{0.45\textwidth}
\begin{center}
\includegraphics[width=\textwidth]{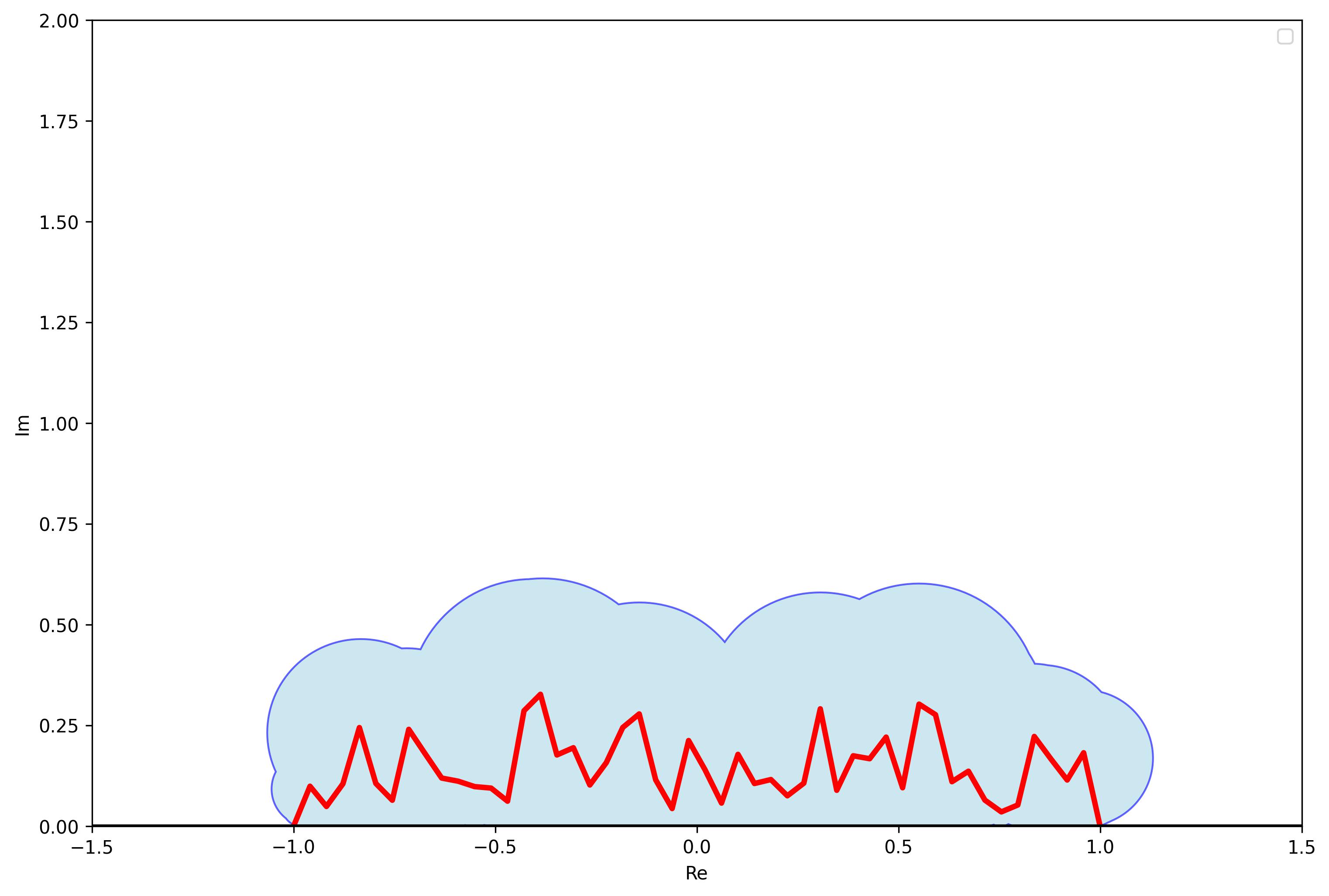}
\end{center}
\caption{}
\end{subfigure}

\caption{\label{fig::hsiz}  
Illustration for $\mathrm{hsiz}(\gamma)$ region $\cup_{x+\mathfrak{i} y\in \gamma} B_y(x+\mathfrak{i} y)$. The red curve indicates $\gamma$ and the green region indicates the region for $\mathrm{hsiz}(\gamma)$.}
\end{figure}

\begin{theorem}\label{thm::multichordalSLE_boundaryarm}
Fix $\kappa\in (0,4], n\ge 1$ and $\alpha\in\mathrm{LP}_n$. Fix $\boldsymbol{x}=(x_1, \ldots, x_{2n})\in\mathcal{X}_{2n}$. Suppose $(\gamma^1, \ldots, \gamma^n)\sim\mathbb{Q}_{\alpha}(\boldsymbol{x})$ is multichordal $\mathrm{SLE}_{\alpha}$ in $(\mathbb{H}; \boldsymbol{x})$ associated to $\alpha$. 
Then\footnote{For two functions $F_1(\boldsymbol{x}; R)$ and $F_2(\boldsymbol{x}; R)$, we write $F_1(\boldsymbol{x}; R)\asymp F_2(\boldsymbol{x}; R)$ as $R/|\boldsymbol{x}|\to\infty$ if there exist constants $C\in (1, \infty)$ and $R_0>0$ such that $C^{-1}\le \frac{F_1(\boldsymbol{x}; R)}{F_2(\boldsymbol{x}; R)}\le C$ for all $R\ge R_0|\boldsymbol{x}|$.} 
\begin{equation}\label{eqn::SLE_halfplane_Green}
\mathbb{Q}_{\alpha}(\boldsymbol{x})\left[\mathrm{hsiz}(\gamma^j)>R^2,\forall 1\le j\le n\right]\asymp G_{\alpha}(\boldsymbol{x})R^{-A_{2n}^+}, \qquad \text{as }R/|\boldsymbol{x}|\to\infty,
\end{equation} 
where $A_{2n}^+$ is the half-plane arm exponent for SLE~\eqref{eqn::SLE_halfplane_arm}
and $G_{\alpha}$ is Green's function in~\eqref{eqn::Green} and the implicit constants in $\asymp$ depend only on $\kappa, n, \alpha$. 
\end{theorem}

The proof of Theorem~\ref{thm::multichordalSLE_boundaryarm} relies on Proposition~\ref{prop::cvg_DysonBM} and the comparison between $\mathrm{hsiz}$ and half-plane capacity derived in~\cite{LLN09}. 
The exponent $A_{2n}^+$~\eqref{eqn::SLE_halfplane_arm} appeared in various places about boundary behavior of SLE curves~\cite{LawlerMinkowskiSLERealLine, MillerWuSLEIntersection, WuZhanSLEBoundaryArmExponents, ZhanGreen2SLEboundary}, and it was usually calculated using SLE martingales. 
Proposition~\ref{prop::cvg_DysonBM} gives a new derivation of the exponent $A_{2n}^+$ using the asymptotic of Dyson Brownian motion transition density. 
This exponent gives the corresponding arm exponents for critical lattice models~\cite{SmirnovWernerCriticalExponents, WuAlternatingArmIsing}. 
This exponent can also be predicted using conformal field theory~\cite{DuplantierConformalFractalGeometry, DuplantierGuttmannConfinedPolymerNetworks}: 
$A_{2n}^+=h_{1, 2n+1}(\kappa)$ 
where $h_{1,s+1}(\kappa)$ is Kac conformal weight
\[h_{1,s+1}(\kappa)=\frac{s(s+2)}{\kappa}-\frac{s}{2}.\]

The organization of the article and relation to other literature are summarized as follows. 
We introduce multi-time local martingale in Section~\ref{sec::multitime_mart}. This is the main tool in the proof of Theorem~\ref{thm::transitiondensity}. 
We prove Theorem~\ref{thm::transitiondensity} in Section~\ref{sec::commontime_mart}. 
We first introduce common-time parameter in Section~\ref{subsec::commontime_mart}. 
Analogous common-time parameter in radial coordinate can be found in~\cite{HealeyLawlerNSidedRadialSLE}. 
We then introduce Dyson Brownian motion in Section~\ref{subsec::DysonBM}. The most relevant conclusion for Dyson Brownian motion is the asymptotics of its transition density proved in~\cite{RoslerGeneralizedHermiteDunklOperators}, summarized in Lemma~\ref{lem::density*}. We complete the proof of Theorem~\ref{thm::transitiondensity} in Section~\ref{subsec::transitiondensity_proof}. 
We complete the proof of Proposition~\ref{prop::cvg_DysonBM} in Section~\ref{subsec::lifetime}. Its proof relies on Theorem~\ref{thm::transitiondensity} and Lemma~\ref{lem::density*}. 
We apply Theorem~\ref{thm::transitiondensity} for level lines of Gaussian free field in Section~\ref{sec::GFF_levellines}.
This is the case when $\kappa=4$ and the corresponding Dyson Brownian motion is of parameter $\beta=2$. In this case, the partition function~\eqref{eqn::GFF_partitionfunction} and the SDEs~\eqref{eqn::GFFlevellines_SDE} and the Green's function~\eqref{eqn::Greenfunction_GFF} are all explicit, see Proposition~\ref{prop::GFF_DysonBM}. 
We prove Theorem~\ref{thm::multichordalSLE_boundaryarm} in Section~\ref{sec::NSLE_estimate}. 

\paragraph*{Acknowledgments.}
We thank Zhonggen Su who suggested this topic to H.W. and provided helpful discussions. We thank Vadim Gorin for helpful discussion and in particular, for pointing out useful references in Lemma~\ref{lem::density*}. We thank Mo Chen for helpful discussion and in particular, for pointing out useful references for the geometric interpretation of half plane capacity. H.W. is supported by New Cornerstone Investigator Program 100001127. H.W. is partly affiliated at Yanqi Lake Beijing Institute of Mathematical Sciences and Applications, Beijing, China.

\section{Multi-time local martingales}
\label{sec::multitime_mart}
\subsection{Preliminaries on SLE}
\label{subsec::pre_Loewner}

\paragraph*{Half-plane capacity.}
An $\mathbb{H}$-hull is a relatively closed subset $K$ of $\mathbb{H}$ such that $K$ is bounded and $\mathbb{H}\setminus K$ is simply connected. 
Let $g_K$ be
the conformal map from $\mathbb{H}\setminus K$ onto $\mathbb{H}$ with the normalization $\lim_{z\to\infty}|g_K(z)-z|=0$. We call $g_K$ the mapping-out function of $K$. Note that there exists a constant $c_K\ge 0$ such that 
\[g_K(z)=z+\frac{2c_K}{z}+o(1/|z|), \qquad \text{as }z\to\infty.\]
The constant $c_K$ is called the half-plane capacity and denoted by $\mathrm{hcap}(K)$.  

\paragraph*{Loewner chain and SLE.}
Fix $x\in\mathbb{R}$ and $T>0$. Let $\eta: [0,T]\to\overline{\mathbb{H}}$ be a continuous simple curve such that $\eta_0=x$ and $\eta_{(0,T)}\subset\mathbb{H}$. For each $t\in[0,T]$, let $g_t$ be the mapping-out function of $\eta_{[0,t]}$. We say that $\eta$ is parameterized by the half-plane capacity if, for all $t\in [0,T]$, 
\[g_t(z)=z+\frac{2t}{z}+o(1/|z|), \qquad\text{as }z\to\infty.\]
These maps $\{g_t: t\in [0,T]\}$ solve the Loewner equation
\begin{equation*}
\partial_t g_t(z)=\frac{2}{g_t(z)-W_t}, \qquad g_0(z)=z,
\end{equation*}
where $W: [0,T]\to \mathbb{R}$ is a continuous function called the driving function of $\eta$. 

Fix $\kappa>0$ and $x\in\mathbb{R}$. The process $\mathrm{SLE}_{\kappa}$ in $(\mathbb{H}; x, \infty)$ is defined as Loewner chain with driving function $W_t=\sqrt{\kappa}B_t+x$ where $B_t$ is standard one-dimensional Brownian motion. When $\kappa\in (0,4]$, it is almost surely a continuous simple curve. For a $2$-polygon $(\Omega; x, y)$, let $\varphi: \Omega\to\mathbb{H}$ be a conformal map with $\varphi(y)=\infty$. $\mathrm{SLE}_{\kappa}$ in $(\Omega; x, y)$ is defined as the image of $\mathrm{SLE}_{\kappa}$ in $(\mathbb{H}; \varphi(x), \infty)$ under $\varphi^{-1}$.

\paragraph*{SLE with force points.} Fix $\kappa>0$ and $\ell \ge 1, r\ge 1$. Denote $\boldsymbol{x}^L=(x^L_{\ell}, \ldots, x^L_1), \boldsymbol{x}^R=(x^R_1, \ldots, x^R_r)$ with $x^L_{\ell}<\cdots<x^L_1<0<x_1^R<\cdots<x_r^R$ and $\boldsymbol{\rho}^L=(\rho^L_{\ell}, \ldots, \rho^L_1)\in\mathbb{R}^{\ell}, \boldsymbol{\rho}^R=(\rho^R_1, \ldots, \rho^R_r)\in \mathbb{R}^r$. The process $\mathrm{SLE}_{\kappa}(\boldsymbol{\rho}^L; \boldsymbol{\rho}^R)$ in $\mathbb{H}$ from $0$ to $\infty$ with force points $(\boldsymbol{x}^L; \boldsymbol{x}^R)$ is defined as the Loewner chain with driving function $W_t$ solves the SDE
\begin{align}
\begin{cases}
\displaystyle\mathrm{d} W_t=\sqrt{\kappa}\mathrm{d} B_t+\sum_{j=1}^{\ell}\frac{\rho_j^L\mathrm{d} t}{W_t-V_t^{L, j}}+\sum_{j=1}^{r}\frac{\rho_j^R\mathrm{d} t}{W_t-V_t^{R, j}}, \qquad W_0=0;\\
\displaystyle\mathrm{d} V_t^{L, j}=\frac{2\mathrm{d} t}{V_t^{L, j}-W_t},\qquad V_0^{L, j}=x_j^L, \quad 1\le j\le \ell; \\
\displaystyle\mathrm{d} V_t^{R, j}=\frac{2\mathrm{d} t}{V_t^{R,j}-W_t}, \qquad V_0^{R, j}=x_j^R, \quad 1\le j\le r. 
\end{cases}
\end{align}
For a $(\ell+r+2)$-polygon $(\Omega; \boldsymbol{x}^L, x, \boldsymbol{x}^R, y)$ with $\boldsymbol{x}^L=(x^L_{\ell}, \ldots, x^L_1), \boldsymbol{x}^R=(x^R_1, \ldots, x^R_r)$, let $\varphi:\Omega\to\mathbb{H}$ be a conformal map with $\varphi(y)=\infty$. 
$\mathrm{SLE}_{\kappa}(\boldsymbol{\rho}^L; \boldsymbol{\rho}^R)$ in $(\Omega; \boldsymbol{x}^L, x, \boldsymbol{x}^R, y)$ is defined as the image of $\mathrm{SLE}_{\kappa}(\boldsymbol{\rho}^L; \boldsymbol{\rho}^R)$ in $\mathbb{H}$ from $\varphi(x)$ to $\infty$ with force points $(\varphi(\boldsymbol{x}^L); \varphi(\boldsymbol{x}^R))$ under $\varphi^{-1}$. 

\subsection{Local martingale for multichordal SLE}
\label{subsec::multitime_mart}

\paragraph*{Multi-time parameter.}
Fix $p\ge 2$ and $\boldsymbol{x}=(x_1, \ldots, x_p)\in\mathcal{X}_p$. Consider a $p$-tuple $\boldsymbol{\eta}_{\boldsymbol{t}}=(\eta^1_{t_1}, \ldots, \eta^p_{t_p})$ of simple curves in $\mathbb{H}$, parameterized by $\boldsymbol{t}=(t_1, \ldots, t_p)\in[0,\infty)^p$, such that $\eta^j_0=x_j$ for each $j$ and the segments $\{\eta^j_{[0,t_j]}\}_{1\le j\le p}$ are disjoint. We define the following normalized conformal transformations:
\begin{itemize}
\item $g_{t_j}^j$ is the conformal map from $\mathbb{H}\setminus\eta^j_{[0,t_j]}$ onto $\mathbb{H}$ with $\lim_{z\to\infty}|g_{t_j}^j(z)-z|=0$, $1\le j\le p$.
\item $g_{\boldsymbol{t}}$ is the conformal map from $\mathbb{H}\setminus\cup_{j=1}^p \eta^j_{[0,t_j]}$ onto $\mathbb{H}$ with $\lim_{z\to\infty}|g_{\boldsymbol{t}}(z)-z|=0$.
\item $g_{\boldsymbol{t}, j}$ is the conformal map from $\mathbb{H}\setminus g_{t_j}^j\left(\cup_{i\neq j}\eta^i_{[0,t_i]}\right)$ onto $\mathbb{H}$ with $\lim_{z\to\infty}|g_{\boldsymbol{t}, j}(z)-z|=0$, $1\le j\le p$. 
\end{itemize}
Using these notations, we have $g_{\boldsymbol{t}}=g_{\boldsymbol{t}, j}\circ g_{t_j}^j$ for $1\le j\le p$. We say that the $p$-tuple $\boldsymbol{\eta}_{\boldsymbol{t}}$ of curves has $p$-time parameter if $g^{j}$ is parameterized by the half-plane capacity: 
\begin{equation}\label{eqn::multitime_def}
g_{t_j}^j(z)=z+\frac{2t_j}{z}+o(1/|z|),\qquad \text{as }z\to\infty.
\end{equation}
Denote by $W^j$ the driving function of each $\eta^j$, and we define the driving function of the $p$-tuple $\boldsymbol{\eta}_{\boldsymbol{t}}$, started at $\boldsymbol{X}_0=(x_1, \ldots, x_p)$, by
\begin{equation}\label{eqn::multitime_driving}
\boldsymbol{X}_{\boldsymbol{t}}=(X^1_{\boldsymbol{t}}, \ldots, X^p_{\boldsymbol{t}}), \qquad \text{with }X^j_{\boldsymbol{t}}=g_{\boldsymbol{t}, j}(W^j_{t_j}), \qquad \text{for }1\le j\le p.
\end{equation}

\begin{proposition}\label{prop::multitime_mart}
Fix $p\ge 2$ and $\boldsymbol{x}=(x_1, \ldots, x_p)\in\mathcal{X}_p$. We fix 
\begin{equation}\label{eqn::parametersbc}
\lambda\in\mathbb{R}, \qquad\kappa\in (0,4],\qquad \mathfrak{b}=\frac{6-\kappa}{2\kappa} \qquad\text{and}\qquad\mathfrak{c}=\frac{(6-\kappa)(3\kappa-8)}{2\kappa}.
\end{equation}
For each $j\in\{1, \ldots, p\}$, let $\eta^j$ be $\mathrm{SLE}_{\kappa}$ in $(\mathbb{H}; x_j,\infty)$ and let $\mathbb{P}_p$ be the probability measure on $\boldsymbol{\eta}=(\eta^1, \ldots, \eta^p)$ under which the curves are independent. We parameterize $\boldsymbol{\eta}$ by $p$-time parameter $\boldsymbol{t}$ and let $\boldsymbol{X}_{\boldsymbol{t}}$ denote the driving function as in~\eqref{eqn::multitime_driving} and denote $K_{\boldsymbol{t}}=\cup_{j=1}^p\eta^j_{[0,t_j]}$. For a function $\mathcal{Z}: \mathcal{X}_p\to \mathbb{R}_{>0}$, define 
\begin{equation}\label{eqn::multitime_mart_general}
M_{\boldsymbol{t}}(\mathcal{Z})=\mathbb{1}_{\mathcal{E}_{\emptyset}(\boldsymbol{\eta}_{\boldsymbol{t}})} \mathcal{Z}(\boldsymbol{X}_{\boldsymbol{t}})\prod_{j=1}^p g'_{\boldsymbol{t},j}(W^j_{t_j})^{\mathfrak{b}}\times \exp\left(\frac{\mathfrak{c}}{2}\mathfrak{m}_{\boldsymbol{t}}-\frac{\lambda}{\kappa}\mathrm{hcap}(K_{\boldsymbol{t}})\right), 
\end{equation}
where $\mathcal{E}_{\emptyset}(\boldsymbol{\eta}_{\boldsymbol{t}})=\{\eta^i_{[0,t_i]}\cap\eta^j_{[0,t_j]}=\emptyset, \forall i\neq j\}$ is the event that different curves are disjoint, and $\mathfrak{m}_{\boldsymbol{t}}$ is the unique potential solving the exact differential equation (see Lemma~\ref{lem::mt_exact})
\begin{equation}\label{eqn::mt_def}
\mathrm{d} \mathfrak{m}_{\boldsymbol{t}}=\sum_{j=1}^p \partial_{t_j}\mathfrak{m}_{\boldsymbol{t}}\mathrm{d} t_j=\sum_{j=1}^p -\frac{1}{3}\mathcal{S} g_{\boldsymbol{t}, j}(W^j_{t_j})\mathrm{d} t_j,\qquad \mathfrak{m}_{\boldsymbol{0}}=0,
\end{equation}
with $\mathcal{S} g=\frac{g'''}{g'}-\frac{3}{2}\left(\frac{g''}{g'}\right)^2$ denoting the Schwarzian derivative of a function $g$. Then, the process $M_{\boldsymbol{t}}(\mathcal{Z})$ is $p$-time parameter local martingale under $\mathbb{P}_p$ if and only if $\mathcal{Z}$ is smooth and satisfies the following system of BPZ equations: for each $\boldsymbol{x}=(x_1, \ldots, x_p)\in\mathcal{X}_p$, for all $1\le j\le p$, 
\begin{equation}\label{eqn::BPZ_general}
\left[\frac{\kappa}{2}\partial_j^2+\sum_{\ell\neq j}\left(\frac{2}{(x_{\ell}-x_j)}\partial_{\ell}-\frac{2\mathfrak{b}}{(x_{\ell}-x_j)^2}\right)\right]\mathcal{Z}=\frac{\lambda}{\kappa}\mathcal{Z}. 
\end{equation}
\end{proposition}
In the present article, we only use the conclusion when $\lambda=0$. But the general form~\eqref{eqn::multitime_mart_general} for $\lambda\in\mathbb{R}$ will be used in future work~\cite{ChenHuangWuBPZChordal}.

\begin{proof}
Using It\^o's formula, we obtain
\begin{equation}\label{eqn::multitime_mart_aux1}
	\mathrm{d} X_{\boldsymbol{t}}^j=g'_{\boldsymbol{t},j}(W_{t_j}^j) \mathrm{d} W_{t_j}^j -\kappa \mathfrak{b} g''_{\boldsymbol{t},j}(W_{t_j}^j) \mathrm{d} t_j + \sum_{\ell\neq j} \frac{2}{X_{\boldsymbol{t}}^j-X_{\boldsymbol{t}}^\ell} \left( g'_{\boldsymbol{t},\ell}(W_{t_\ell}^\ell) \right)^2 \mathrm{d} t_\ell,
\end{equation}
and
\begin{align}\label{eqn::multitime_mart_aux2}
\begin{split}
	\frac{\mathrm{d} g'_{\boldsymbol{t},j}(W_{t_j}^j)}{g'_{\boldsymbol{t},j}(W_{t_j}^j)}= & \frac{g''_{\boldsymbol{t},j}(W_{t_j}^j)}{g'_{\boldsymbol{t},j}(W_{t_j}^j)} \mathrm{d} W_{t_j}^j  - \sum_{\ell\neq j} \frac{2}{\left( X_{\boldsymbol{t}}^j-X_{\boldsymbol{t}}^\ell \right)^2} 	\left( g'_{\boldsymbol{t},\ell}(W_{t_{\ell}}^\ell) \right)^2 \mathrm{d} t_\ell\\
	&+ \left( \left( \frac{3\kappa-8}{6} \right)  \frac{g'''_{\boldsymbol{t},j}(W_{t_j}^j)}{g'_{\boldsymbol{t},j}(W_{t_j}^j)} + \frac{1}{2} \left( \frac{g''_{\boldsymbol{t},j}(W_{t_j}^j)}{g'_{\boldsymbol{t},j}(W_{t_j}^j)} \right)^2 \right) \mathrm{d} t_j.
\end{split}
\end{align}
Moreover, by the additivity of half-plane capacity, we have
\begin{equation} \label{eqn::multitime_mart_aux2'}
	\mathrm{d}\,\mathrm{hcap}(K_{\boldsymbol{t}})=\sum_{j=1}^p \left(g'_{\boldsymbol{t},j}(W^j_{t_j})\right)^2\mathrm{d} t_j.
\end{equation}
We first assume $\mathcal{Z}\in C^2(\mathcal{X}_p)$ and apply It\^{o}'s formula to obtain
\begin{align} \label{eqn::multitime_mart_aux3}
\begin{split}
	\frac{\mathrm{d} M_{\boldsymbol{t}}(\mathcal{Z})}{M_{\boldsymbol{t}}(\mathcal{Z})}= & \sum_{j=1}^p \left( \mathfrak{b} \frac{\mathrm{d} g'_{\boldsymbol{t},j}(W_{t_j}^j)}{g'_{\boldsymbol{t},j}(W_{t_j}^j)} + \frac{\kappa \mathfrak{b}(\mathfrak{b}-1)}{2} \left( \frac{g''_{\boldsymbol{t},j}(W_{t_j}^j)}{g'_{\boldsymbol{t},j}(W_{t_j}^j)} \right)^2 \mathrm{d} t_j \right)  + \frac{\mathfrak{c}}{2} \mathrm{d} \mathfrak{m}_{\boldsymbol{t}} -\frac{\lambda}{\kappa}\mathrm{d}\,\mathrm{hcap}(K_{\boldsymbol{t}})\\
	& + \sum_{j=1}^p \left(  \frac{\partial_j \mathcal{Z} (\boldsymbol{X_t})}{\mathcal{Z} (\boldsymbol{X_t})} \mathrm{d} X_{\boldsymbol{t}}^j + \frac{\kappa}{2} \frac{\partial_j^2 \mathcal{Z}(\boldsymbol{X_t})}{\mathcal{Z}(\boldsymbol{X_t})} \left( g'_{\boldsymbol{t},j}(W_{t_j}^j) \right)^2  \mathrm{d} t_j \right) + \kappa \mathfrak{b} \sum_{j=1}^{p} \frac{\partial_j \mathcal{Z}(\boldsymbol{X_t})}{\mathcal{Z}(\boldsymbol{X_t})} g''_{\boldsymbol{t},j}(W_{t_j}^j) \mathrm{d} t_j.
\end{split}
\end{align}
Plugging~\eqref{eqn::mt_def},~\eqref{eqn::multitime_mart_aux1},~\eqref{eqn::multitime_mart_aux2} and~\eqref{eqn::multitime_mart_aux2'} into~\eqref{eqn::multitime_mart_aux3} we obtain
\begin{align} \label{eqn::multitime_mart_aux4}
\begin{split}
	\frac{\mathrm{d} M_{\boldsymbol{t}}(\mathcal{Z})}{M_{\boldsymbol{t}}(\mathcal{Z})}= & \sum_{j=1}^p \left( \frac{\partial_j \mathcal{Z}(\boldsymbol{X_t})}{\mathcal{Z}(\boldsymbol{X_t})} g'_{\boldsymbol{t},j}(W_{t_j}^j) + \mathfrak{b} \frac{g''_{\boldsymbol{t},j}(W_{t_j}^j)}{g'_{\boldsymbol{t},j}(W_{t_j}^j)} \right) \mathrm{d} W_{t_j}^j \\
	& + \sum_{j=1}^p \left( g'_{\boldsymbol{t},j}(W_{t_j}^j) \right)^2 \left( \frac{\kappa}{2} \frac{\partial_j^2 \mathcal{Z}(\boldsymbol{X_t})}{\mathcal{Z}(\boldsymbol{X_t})} + \sum_{\ell\neq j} \left( \frac{2}{(X_{\boldsymbol{t}}^\ell-X_{\boldsymbol{t}}^j)} \frac{\partial_\ell \mathcal{Z}(\boldsymbol{X_t})}{\mathcal{Z}(\boldsymbol{X_t})} - \frac{2\mathfrak{b}}{(X_{\boldsymbol{t}}^\ell-X_{\boldsymbol{t}}^j)^2} \right) -\frac{\lambda}{\kappa}\right) \mathrm{d} t_j.
\end{split}
\end{align}
From this expression, we see that $M_{\boldsymbol{t}}(\mathcal{Z})$ is a $p$-time parameter local martingale if and only if all the drift terms vanish, which happens precisely when $\mathcal{Z}$ satisfies the system of BPZ equations~\eqref{eqn::BPZ_general}.

To lift the assumption that $\mathcal{Z}$ is $C^2$, the above analysis tells that $M_{\boldsymbol{t}}(\mathcal{Z})$ is a $p$-time parameter local martingale if and only if $\mathcal{Z}$ satisfies the system of BPZ equations~\eqref{eqn::BPZ_general} as a weak solution. Since the differential operators in~\eqref{eqn::BPZ_general} are hypoelliptic, weak solutions are strong solutions and in particular are smooth, see~\cite{KarrilaWiitasaariMartObservable}. This concludes the proof.
\end{proof}

In the following two lemmas, we address the term $\mathfrak{m}_{\boldsymbol{t}}$ defined through~\eqref{eqn::mt_def}. In Lemma~\ref{lem::mt_exact}, we show that the differential equation~\eqref{eqn::mt_def} is exact and hence $\mathfrak{m}_{\boldsymbol{t}}$ is well-defined. In Lemma~\ref{lem::mt_blm}, we show that $\mathfrak{m}_{\boldsymbol{t}}$ can be interpreted as Brownian loop measure and hence it is finite. 

\begin{lemma}\label{lem::mt_exact}
The differential equation~\eqref{eqn::mt_def} is exact and has a unique solution $\mathfrak{m}_{\boldsymbol{t}}$. 
\end{lemma}

\begin{proof}
By a direct but tedious calculation, we have
\begin{equation}\label{eqn::mt_def_aux}
\partial_{t_i}\left(-\frac{1}{3}\mathcal{S} g_{\boldsymbol{t}, j}(W^j_{t_j})\right)=\frac{4\left( g'_{\boldsymbol{t},i}(W_{t_i}^i)  g'_{\boldsymbol{t},j}(W_{t_j}^j) \right)^2}{\left( X_{\boldsymbol{t}}^j-X_{\boldsymbol{t}}^i \right)^4}. 
\end{equation}
The RHS of~\eqref{eqn::mt_def_aux} is symmetric with respect to the exchange $i\leftrightarrow j$. Thus, 
\[\partial_{t_i}\left(-\frac{1}{3}\mathcal{S} g_{\boldsymbol{t}, j}(W^j_{t_j})\right)=\partial_{t_j}\left(-\frac{1}{3}\mathcal{S} g_{\boldsymbol{t}, i}(W^i_{t_i})\right).\]
Then there exists a unique potential $\mathfrak{m}_{\boldsymbol{t}}$ such that 
\[\partial_{t_j}\mathfrak{m}_{\boldsymbol{t}}=-\frac{1}{3}\mathcal{S} g_{\boldsymbol{t}, j}(W^j_{t_j}), \qquad 1\le j\le p,\]
and $\mathfrak{m}_{\boldsymbol{0}}=0$. 
\end{proof}

\paragraph*{Brownian loop measure.}
Brownian loop measure $\mathfrak{m}^\mathrm{loop}$ is a $\sigma$-finite measure on planar unrooted Brownian loops
--- see ~\cite{LawlerWernerBrownianLoopsoup} for its definition and properties. 
While the total mass of $\mathfrak{m}^\mathrm{loop}$ is infinite, the mass on macroscopic loops is finite: 
if $\Omega$ is a domain and $K_1, K_2 \subset \overline{\Omega}$ are two disjoint compact subsets, 
then the total mass $\mathfrak{m}(\Omega; K_1, K_2)$ of Brownian loops that stay in $\Omega$ and intersect both $K_1$ and $K_2$ is finite.
In general, for $p\ge 2$ disjoint compact subsets $K_1, \ldots, K_p$ of $\overline{\Omega}$, we denote
\begin{align}\label{eqn::blm_def}
\mathfrak{m}(\Omega; K_1, \ldots, K_p) := \sum_{j=2}^p \mathfrak{m}^{\mathrm{loop}} \big[ \ell\subset\Omega: \ell\cap K_i\neq \emptyset\textnormal{ for at least }j \textnormal{ of the }i\in\{1, \ldots, p\} \big].
\end{align}
See~\cite{LawlerPartitionFunctionsSLE} and~\cite{PeltolaWangSLELDP} for more properties and~\cite{DubedatEulerIntegralsCommutingSLEs, DubedatCommutationSLE, KozdronLawlerMultipleSLEs, PeltolaWuGlobalMultipleSLEs} 
for alternative forms for~\eqref{eqn::blm_def}.

\begin{lemma}\label{lem::mt_blm}
The solution to~\eqref{eqn::mt_def} 
can be described in terms of Brownian loop measure as
\begin{align}\label{eqn::mt_blm}
\mathfrak{m}_{\boldsymbol{t}} = \mathfrak{m} \left(\mathbb{H};\eta_{[0,t_1]}^{1},\ldots,\eta_{[0,t_p]}^{p} \right).
\end{align}
Consequently, $\mathfrak{m}_{\boldsymbol{t}}$ is finite as long as $\eta_{[0,t_1]}^{1},\ldots,\eta_{[0,t_p]}^{p}$ are disjoint. 
\end{lemma}

\begin{proof}
Denote by $\tilde{\mathfrak{m}}_{\boldsymbol{t}}$ the RHS of~\eqref{eqn::mt_blm}. Note that both sides of~\eqref{eqn::mt_blm} satisfy the same normalization $\mathfrak{m}_{\boldsymbol{0}}=\tilde{\mathfrak{m}}_{\boldsymbol{0}}=0$. Hence, it suffices to show that $\partial_{t_j}\mathfrak{m}_{\boldsymbol{t}}=\partial_{t_j} \tilde{\mathfrak{m}}_{\boldsymbol{t}}$ for each $j$. 
Due to symmetry, it suffices to show $\partial_{t_1}\mathfrak{m}_{\boldsymbol{t}}=\partial_{t_1} \tilde{\mathfrak{m}}_{\boldsymbol{t}}$. 
From~\eqref{eqn::blm_def} we obtain
\begin{align*}
\begin{split}
	& \mathfrak{m} \left(\mathbb{H};\eta_{[0,t_1]}^{1},\ldots,\eta_{[0,t_p]}^{p} \right) \\
	= & \sum_{j=1}^{p-1} \mathfrak{m}^{\mathrm{loop}} \big[ \ell \subset \mathbb{H}: \ell\cap \eta_{[0,t_1]}^{1}\neq \emptyset, \text{ and } \ell\cap \eta_{[0,t_i]}^{i}\neq \emptyset \textnormal{ for at least }j \textnormal{ of the }i\in\{2, \ldots, p\} \big] \\
	& + \sum_{j=2}^{p-1} \mathfrak{m}^{\mathrm{loop}} \big[ \ell \subset \mathbb{H} \setminus \eta_{[0,t_1]}^{1}: \ell\cap \eta_{[0,t_i]}^{i}\neq \emptyset \textnormal{ for at least }j \textnormal{ of the }i\in\{2, \ldots, p\} \big] \\
	= & \mathfrak{m} \left(\mathbb{H};\eta_{[0,t_1]}^{1},\cup_{j=2}^p \eta_{[0,t_j]}^{j} \right) + \mathfrak{m} \left(\mathbb{H};\eta_{[0,t_2]}^{2},\ldots, \eta_{[0,t_p]}^{p} \right).
\end{split}
\end{align*}
Using the relation between Brownian loop measure and Schwarzian derivative $\mathcal{S}$ (see~\cite[Eq.~(7.3)]{LawlerSchrammWernerConformalRestriction} and~\cite[Proposition~8]{LawlerWernerBrownianLoopsoup}, or the equation before~\cite[Eq.~(10)]{KozdronLawlerMultipleSLEs}), we obtain
\begin{equation*}
	\partial_{t_1} \tilde{\mathfrak{m}}_{\boldsymbol{t}}=\partial_{t_1} \mathfrak{m} \left(\mathbb{H};\eta_{[0,t_1]}^{1},\cup_{j=2}^p \eta_{[0,t_j]}^{j} \right) =-\frac{1}{3}\mathcal{S} g_{\boldsymbol{t}, 1}(W^1_{t_1}), 
\end{equation*}
as desired. 
\end{proof}

\subsection{Half-watermelon SLE and its partition functions}

\begin{definition}[Half-watermelon SLE]\label{def::halfwatermelonSLE}
Fix $\kappa\in (0,4]$ and $p\ge 2$. Let $(\Omega; \boldsymbol{x}, y) = (\Omega; x_1, \ldots, x_p, y)$ be a $(p+1)$-polygon. 
Consider curves $\eta^{j}\in\mathfrak{X}(\Omega; x_j, y)$ in $\Omega$ from $x_j$ to $y$ for each $j$; and suppose $\eta^{i}\cap\eta^{j}=\{y\}$ for $i\neq j$. 
We call the law $\mathbb{Q}_{\includegraphics[scale=0.5]{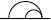}_{p}}(\Omega; \boldsymbol{x}; y)$ of $\boldsymbol{\eta} = (\eta^{1}, \ldots, \eta^{p})$ half-$p$-watermelon $\mathrm{SLE}_{\kappa}$ in $(\Omega; \boldsymbol{x}, y)$ from $\boldsymbol{x}$ to $y$ if 
it satisfies the resampling property:  
for each $j\in \{1, \ldots, p\}$, the conditional law of $\eta^{j}$ given $\{\eta^{i}: i\neq j\}$ is chordal $\mathrm{SLE}_{\kappa}$ from $x_j$ to $y$ 
in the connected component $\Omega_j$ of $\Omega\setminus\cup_{i\neq j}\eta^{i}$ having $x_j$ on its boundary.
In particular, when $(\Omega; \boldsymbol{x}, y)=(\mathbb{H}; \boldsymbol{x}, \infty)$, we denote by $\mathbb{Q}_{p}(\boldsymbol{x})$ the law of half-$p$-watermelon $\mathrm{SLE}_{\kappa}$ in $(\mathbb{H}; \boldsymbol{x}, \infty)$. 
\end{definition}

The existence and uniqueness of half-$p$-watermelon $\mathrm{SLE}_{\kappa}$ was argued in~\cite[Section~4]{MillerSheffieldIG2}. For half-$p$-watermelon $\mathrm{SLE}_{\kappa}$ $\boldsymbol{\eta} = (\eta^{1}, \ldots, \eta^{p}) \sim \mathbb{Q}_{\includegraphics[scale=0.5]{figures/link4fusion}_{p}}(\Omega; \boldsymbol{x}, y)$, 
the marginal law of each $\eta^{j}$ is chordal $\mathrm{SLE}_{\kappa}(2, \ldots, 2)$ in $\Omega$ from $x_j$ to $y$ with force points $(x_1, \ldots, x_{j-1}; x_{j+1}, \ldots, x_p)$.  
The partition function for half-watermelon SLE in $(\mathbb{H}; \boldsymbol{x}, \infty)$ is given by
\begin{align} \label{eqn::LZfusion_H}
\mathcal{Z}_{p}(\boldsymbol x) := \; & \prod_{1\leq i<j\leq p}(x_{j}-x_{i})^{\frac{2}{\kappa}}. 
\end{align}

\begin{lemma}\label{lem::halfwatermelon_mart}
Fix $\kappa\in (0,4], p\ge 2$ and $\boldsymbol{x}=(x_1, \ldots, x_p)\in\mathcal{X}_p$. 
For each $j\in\{1, \ldots, p\}$, let $\eta^j$ be $\mathrm{SLE}_{\kappa}$ in $(\mathbb{H}; x_j,\infty)$ and let $\mathbb{P}_p$ be the probability measure on $\boldsymbol{\eta}=(\eta^1, \ldots, \eta^p)$ under which the curves are independent. We parameterize $\boldsymbol{\eta}$ by $p$-time parameter $\boldsymbol{t}$ and let $\boldsymbol{X}_{\boldsymbol{t}}$ denote the driving function as in~\eqref{eqn::multitime_driving}. Then the process 
\begin{equation}\label{eqn::halfwatermelon_multitime_mart}
M_{\boldsymbol{t}}(\mathcal{Z}_{p})=\mathbb{1}_{\mathcal{E}_{\emptyset}(\boldsymbol{\eta}_{\boldsymbol{t}})} \prod_{j=1}^p g'_{\boldsymbol{t},j}(W^j_{t_j})^{\mathfrak{b}}\times \exp\left(\frac{\mathfrak{c}}{2}\mathfrak{m}_{\boldsymbol{t}}\right)\times \mathcal{Z}_{p}(\boldsymbol{X}_{\boldsymbol{t}})
\end{equation}
is $p$-time parameter local martingale with respect to $\mathbb{P}_p$ where $\mathcal{E}_{\emptyset}(\boldsymbol{\eta}_{\boldsymbol{t}})$ and $\mathfrak{m}_{\boldsymbol{t}}$ are defined in the same way as in Proposition~\ref{prop::multitime_mart}. Moreover, the law of $\mathbb{P}_p$ weighted by $M_{\boldsymbol{t}}(\mathcal{Z}_{p})$ is the same as half-$p$-watermelon $\mathrm{SLE}_{\kappa}$ in $(\mathbb{H}; \boldsymbol{x}, \infty)$ when restricted to the event $\mathcal{E}_{\emptyset}(\boldsymbol{\eta}_{\boldsymbol{t}})$. 
\end{lemma}
\begin{proof}
One may check by direct calculation that $\mathcal{Z}_{p}$ satisfies the system of BPZ equations~\eqref{eqn::BPZ_general}. Thus, Proposition~\ref{prop::multitime_mart} guarantees that $M_{\boldsymbol{t}}(\mathcal{Z}_{p})$ is $p$-time parameter local martingale with respect to $\mathbb{P}_p$. 

Denote by $\mathbb{P}(\mathcal{Z}_{p})$ the measure obtained by weighting $\mathbb{P}_p$ by $M_{\boldsymbol{t}}(\mathcal{Z}_{p})$. It remains to show that $\mathbb{P}(\mathcal{Z}_{p})$ is the same as $\mathbb{Q}_{p}(\boldsymbol{x})$ the law of half-$p$-watermelon $\mathrm{SLE}_{\kappa}$ in $(\mathbb{H}; \boldsymbol{x}, \infty)$. From~\eqref{eqn::multitime_mart_aux4} we have
\begin{equation} \label{eqn::halfwatermelon_mart_aux1}
	\frac{\mathrm{d} M_{\boldsymbol{t}}(\mathcal{Z}_{p})}{M_{\boldsymbol{t}}(\mathcal{Z}_{p})}=  \sum_{j=1}^p \left( \frac{\partial_j \mathcal{Z}_{p}(\boldsymbol{X_t})}{\mathcal{Z}_{p}(\boldsymbol{X_t})} g'_{\boldsymbol{t},j}(W_{t_j}^j) + \mathfrak{b} \frac{g''_{\boldsymbol{t},j}(W_{t_j}^j)}{g'_{\boldsymbol{t},j}(W_{t_j}^j)} \right) \mathrm{d} W_{t_j}^j.
\end{equation}
Thus Girsanov's theorem implies
\begin{align} \label{eqn::halfwatermelon_mart_aux2}
\begin{split}
	\mathrm{d} X_{\boldsymbol{t}}^j=& \sqrt{\kappa} g'_{\boldsymbol{t},j}(W_{t_j}^j) \mathrm{d} \tilde{B}_{t_j}^j + \kappa \partial_j (\log \mathcal{Z}_{p}) (\boldsymbol{X_t}) \left( g'_{\boldsymbol{t},j}(W_{t_j}^j) \right)^2 \mathrm{d} t_j + \sum_{\ell\neq j} \frac{2}{X_{\boldsymbol{t}}^j-X_{\boldsymbol{t}}^\ell} \left( g'_{\boldsymbol{t},\ell}(W_{t_\ell}^\ell) \right)^2 \mathrm{d} t_\ell \\
	=& \sqrt{\kappa} g'_{\boldsymbol{t},j}(W_{t_j}^j) \mathrm{d} \tilde{B}_{t_j}^j + \sum_{\ell\neq j} \frac{2}{X_{\boldsymbol{t}}^j-X_{\boldsymbol{t}}^\ell} \left(\left( g'_{\boldsymbol{t},j}(W_{t_j}^j) \right)^2 \mathrm{d} t_j + \left( g'_{\boldsymbol{t},\ell}(W_{t_\ell}^\ell) \right)^2 \mathrm{d} t_\ell \right),
\end{split}
\end{align}
where $\tilde{B}_{t_1}^1, \ldots ,\tilde{B}_{t_p}^p$ are independent Brownian motions under $\mathbb{P}(\mathcal{Z}_{p})$. 
Taking into account the variation of the capacity parameterization, under $\mathbb{P}(\mathcal{Z}_{p})$, 
\begin{itemize}
	\item the marginal law of $\eta_{[0,t_1]}^1$ is chordal $\mathrm{SLE}_{\kappa}(2,\ldots,2)$ in $\mathbb{H}$ from $x_1$ to $\infty$ with force points $(x_2,\ldots,x_p)$;
	\item for each $j\in \{1,\ldots,p-1\}$, conditionally on $(\eta_{[0,t_1]}^1,\ldots,\eta_{[0,t_j]}^j)$, the law of the curve $\eta_{[0,t_{j+1}]}^{j+1}$ is a chordal $\mathrm{SLE}_{\kappa}(2,\ldots,2;2,\ldots,2)$ in $\mathbb{H}\setminus \cup_{i=1}^j \eta_{[0,t_i]}^i$ from $x_{j+1}$ to $\infty$ with force points $(\eta_{t_1}^1,\ldots,\eta_{t_j}^j;x_{j+2},\ldots,x_p)$.
\end{itemize}
Since the law of half-$p$-watermelon is unique, and it is the same as the law of $\boldsymbol{\eta}$ under $\mathbb{P}(\mathcal{Z}_{p})$ as desired.
\end{proof}

\subsection{Multichordal SLE and pure partition functions}
\label{subsec::pre_NSLE}

To describe the Loewner chain of multichordal SLE, we introduce pure partition functions. These are the recursive collection $\{\mathcal{Z}_{\alpha} \colon \alpha \in \sqcup_{n\geq 0} \mathrm{LP}_n\}$ 
of functions $\mathcal{Z}_{\alpha} \colon \mathcal{X}_{2n}\to\mathbb{R}$
uniquely determined by the following four properties: 
\begin{itemize}
	\item BPZ equations: for all $ j \in \{1,\ldots,2n\}$, 
	\begin{align}\label{eqn::PDE}
		\left[ 
		\frac{\kappa}{2} \partial_j^2
		+ \sum_{\ell\neq j} \left(\frac{2}{(x_{\ell}-x_{j})} \partial_{\ell}
		- \frac{2\mathfrak{b}}{(x_{\ell}-x_{j})^{2}}\right)\right]
		\mathcal{Z}_\alpha(x_1,\ldots,x_{2n}) =  0.
	\end{align}
	\item M\"{o}bius covariance: for all M\"obius maps $\varphi$ of the upper half-plane $\mathbb{H}$ such that $\varphi(x_{1}) < \cdots < \varphi(x_{2n})$, we have
	\begin{align*}
		\mathcal{Z}_\alpha(x_{1},\ldots,x_{2n}) = 
		\prod_{j=1}^{2n} \varphi'(x_{j})^{\mathfrak{b}} 
		\times \mathcal{Z}_\alpha(\varphi(x_{1}),\ldots,\varphi(x_{2n})).
	\end{align*}
	\item Asymptotics: with $\mathcal{Z}_{\emptyset} \equiv 1$ for the empty link pattern $\emptyset \in \mathrm{LP}_0$, the collection $\{\mathcal{Z}_{\alpha} \colon \alpha\in\mathrm{LP}_n\}$ satisfies the following recursive asymptotics property. Fix $n\ge 1$ and $j \in \{1,2, \ldots, 2n-1 \}$. Then, we have
	\begin{align*}
		\lim_{x_j,x_{j+1}\to\xi} \frac{\mathcal{Z}_{\alpha}(x_1,\ldots, x_{2n})}{ (x_{j+1}-x_j)^{-2\mathfrak{b}} }= 
		\begin{cases}
			\mathcal{Z}_{\alpha/\{j,j+1\}}(x_1, \ldots, x_{j-1}, x_{j+2}, \ldots, x_{2n}), 
			& \quad \text{if }\{j, j+1\}\in\alpha , \\
			0,
			& \quad \text{if }\{j, j+1\} \not\in \alpha ,
		\end{cases}
	\end{align*}
	where $\xi \in (x_{j-1}, x_{j+2})$ (with the convention that $x_0 = -\infty$ and  $x_{2n+1} = +\infty$), and $\alpha/\{k,l\}$ denotes the link pattern in $\mathrm{LP}_{n-1}$ obtained by removing $\{k,l\}$ from $\alpha$ and then relabeling the remaining indices so that they are the first $2(n-1)$ positive integers. 
	\item The functions satisfy the following power-law bound: 
	\begin{align}\label{eqn::PLB_strong}
		0<\mathcal{Z}_{\alpha}(x_1, \ldots, x_{2n}) \le\prod_{\{k,l\}\in\alpha}|x_k-x_l|^{-2\mathfrak{b}}, \quad \text{for all }x_1<\cdots<x_{2n}. 
	\end{align}
\end{itemize}

The uniqueness when $\kappa\in (0,8)$ of such collection of functions were proved in~\cite{FloresKlebanPDE2, ChenHuangWuBPZChordal}. The existence is proved in~\cite{FloresKlebanPDE3, KytolaPeltolaPurePartitionFunctions, PeltolaWuGlobalMultipleSLEs, WuHyperSLE, AngHoldenSunYu2023, FengLiuPeltolaWu2024}. In this article, we focus on $\kappa\le 4$.

\begin{lemma}\label{lem::ppf_mart}
Fix $\kappa\in (0,4], n\ge 2$ and $\alpha\in\mathrm{LP}_n$. Fix $\boldsymbol{x}=(x_1, \ldots, x_{2n})\in\mathcal{X}_{2n}$.  
For each $j\in\{1, \ldots, 2n\}$, let $\eta^j$ be $\mathrm{SLE}_{\kappa}$ in $(\mathbb{H}; x_j,\infty)$ and let $\mathbb{P}_{2n}$ be the probability measure on $\boldsymbol{\eta}=(\eta^1, \ldots, \eta^{2n})$ under which the curves are independent. We parameterize $\boldsymbol{\eta}$ by $2n$-time parameter $\boldsymbol{t}$ and let $\boldsymbol{X}_{\boldsymbol{t}}$ denote the driving function as in~\eqref{eqn::multitime_driving} with $p=2n$. Then the process 
\begin{equation}\label{eqn::ppf_multitime_mart}
M_{\boldsymbol{t}}(\mathcal{Z}_{\alpha})=\mathbb{1}_{\mathcal{E}_{\emptyset}(\boldsymbol{\eta}_{\boldsymbol{t}})} \prod_{j=1}^{2n} g'_{\boldsymbol{t},j}(W^j_{t_j})^{\mathfrak{b}}\times \exp\left(\frac{\mathfrak{c}}{2}\mathfrak{m}_{\boldsymbol{t}}\right)\times \mathcal{Z}_{\alpha}(\boldsymbol{X}_{\boldsymbol{t}})
\end{equation}
is $2n$-time parameter local martingale with respect to $\mathbb{P}_{2n}$ where $\mathcal{E}_{\emptyset}(\boldsymbol{\eta}_{\boldsymbol{t}})$ and $\mathfrak{m}_{\boldsymbol{t}}$ are defined in the same way as in Proposition~\ref{prop::multitime_mart} where $p=2n$. Moreover, the law of $\mathbb{P}_{2n}$ weighted by $M_{\boldsymbol{t}}(\mathcal{Z}_{\alpha})$ is the same as multichordal $\mathrm{SLE}_{\kappa}$ in $(\mathbb{H}; \boldsymbol{x})$ associated to $\alpha$ when restricted to the event $\mathcal{E}_{\emptyset}(\boldsymbol{\eta}_{\boldsymbol{t}})$. 
\end{lemma}

\begin{proof}
As $\mathcal{Z}_{\alpha}$ satisfies BPZ equations~\eqref{eqn::PDE}, Proposition~\ref{prop::multitime_mart} guarantees that $M_{\boldsymbol{t}}(\mathcal{Z}_{\alpha})$ is $2n$-time parameter local martingales with respect to $\mathbb{P}_{2n}$. 

Denote by $\mathbb{P}(\mathcal{Z}_{\alpha})$ the measure obtained by weighting $\mathbb{P}_{2n}$ by $M_{\boldsymbol{t}}(\mathcal{Z}_{\alpha})$. It remains to show that $\mathbb{P}(\mathcal{Z}_{\alpha})$ is the same as $\mathbb{Q}_{\alpha}(\boldsymbol{x})$ the law of multichordal $\mathrm{SLE}_{\kappa}$ in $(\mathbb{H}; \boldsymbol{x})$ associated to $\alpha$. 
From~\eqref{eqn::multitime_mart_aux4} we have
\begin{equation} \label{eqn::ppf_mart_aux1}
	\frac{\mathrm{d} M_{\boldsymbol{t}}(\mathcal{Z}_{\alpha})}{M_{\boldsymbol{t}}(\mathcal{Z}_{\alpha})}=  \sum_{j=1}^p \left( \frac{\partial_j \mathcal{Z}_{\alpha}(\boldsymbol{X_t})}{\mathcal{Z}_{\alpha}(\boldsymbol{X_t})} g'_{\boldsymbol{t},j}(W_{t_j}^j) + \mathfrak{b} \frac{g''_{\boldsymbol{t},j}(W_{t_j}^j)}{g'_{\boldsymbol{t},j}(W_{t_j}^j)} \right) \mathrm{d} W_{t_j}^j.
\end{equation}
Thus Girsanov's theorem implies
\begin{align} \label{eqn::ppf_mart_aux2}
	\begin{split}
		\mathrm{d} X_{\boldsymbol{t}}^j=& \sqrt{\kappa} g'_{\boldsymbol{t},j}(W_{t_j}^j) \mathrm{d} \tilde{B}_{t_j}^j + \kappa \partial_j (\log \mathcal{Z}_{\alpha}) (\boldsymbol{X_t}) \left( g'_{\boldsymbol{t},j}(W_{t_j}^j) \right)^2 \mathrm{d} t_j + \sum_{\ell\neq j} \frac{2}{X_{\boldsymbol{t}}^j-X_{\boldsymbol{t}}^\ell} \left( g'_{\boldsymbol{t},\ell}(W_{t_\ell}^\ell) \right)^2 \mathrm{d} t_\ell,
	\end{split}
\end{align}
where $\tilde{B}_{t_1}^1, \ldots ,\tilde{B}_{t_p}^p$ are independent Brownian motions under $\mathbb{P}(\mathcal{Z}_{\alpha})$. 
The Loewner chains for each curves in the multichordal SLEs are given in~\cite[Proposition~4.10]{PeltolaWuGlobalMultipleSLEs}. Taking into account the variation of the capacity parameterization, under $\mathbb{P}(\mathcal{Z}_{\alpha})$, by~\eqref{eqn::ppf_mart_aux2}, we can obtain the marginal law and conditional law of each curves under $\mathbb{P}(\mathcal{Z}_{\alpha})$ similarly as in the proof of Lemma~\ref{lem::halfwatermelon_mart}. Thus we conclude from~\eqref{eqn::ppf_mart_aux2} that the law of $\boldsymbol{\eta_t}$ under $\mathbb{P}(\mathcal{Z}_{\alpha})$ is local multichordal $\mathrm{SLE}$ in $(\mathbb{H}; \boldsymbol{x})$ associated to $\alpha$. It was proved in~\cite[Theorem~1.3]{PeltolaWuGlobalMultipleSLEs} that global multiple $\mathrm{SLE}$ in $(\mathbb{H}; \boldsymbol{x})$ associated to $\alpha$ coincides with local multichordal $\mathrm{SLE}$ in $(\mathbb{H}; \boldsymbol{x})$ associated to $\alpha$. Thus $\mathbb{P}(\mathcal{Z}_{\alpha})$ is the same as $\mathbb{Q}_{\alpha}(\boldsymbol{x})$ as desired.
\end{proof}

\section{Proof of Theorem~\ref{thm::transitiondensity} and Proposition~\ref{prop::cvg_DysonBM}}
\label{sec::commontime_mart}
\subsection{Common-time parameter}
\label{subsec::commontime_mart}

Fix $p\ge 2$ and $\boldsymbol{x}=(x_1, \ldots, x_p)\in\mathcal{X}_p$. Consider a $p$-tuple $\boldsymbol{\eta}_{\boldsymbol{t}}=(\eta^1_{t_1}, \ldots, \eta^p_{t_p})$ of simple curves in $\mathbb{H}$, parameterized by $\boldsymbol{t}=(t_1, \ldots, t_p)\in[0,\infty)^p$, such that $\eta^j_0=x_j$ for each $j$ and the segments $\{\eta^j_{[0,t_j]}\}_{1\le j\le p}$ are disjoint. Let $g_{\boldsymbol{t}}$ be the  conformal map from $\mathbb{H}\setminus\cup_{j=1}^p \eta^j_{[0,t_j]}$ onto $\mathbb{H}$ with 
\begin{equation}
g_{\boldsymbol{t}}(z)=z+\frac{\aleph_{\boldsymbol{t}}}{z}+o(1/|z|), \qquad \text{as }z\to\infty.
\end{equation}
We say that the $p$-tuple $\boldsymbol{\eta}_{\boldsymbol{t}}$ have common-time parameter $t$ if 
\begin{equation}\label{eqn::commontime_def}
\partial_{t_j}\aleph_{(t, \ldots, t)}=2,\qquad \text{for all }j\in\{1, \ldots, p\}.
\end{equation}
The existence of common-time parameter is proved in~\cite[Lemma~3.2]{HealeyLawlerNSidedRadialSLE} for the radial coordinates. The existence in chordal coordinates can be proved similarly, see Lemma~\ref{lem::commontime_existence}. 

\begin{lemma}\label{lem::commontime_existence}
Common-time parameter exists for continuous disjoint simple curves $\boldsymbol{\eta}_{\boldsymbol{t}}=(\eta^1_{t_1}, \ldots, \eta^p_{t_p})$. Moreover, 
\begin{align}\label{eqn::commontime_control}
t\le t_j(t)\le pt, \qquad 1\le j\le p. 
\end{align}
\end{lemma}
\begin{proof}
A standard calculation gives 
\begin{align*}
\mathrm{d} \aleph_{\boldsymbol{t}}=2\sum_{j=1}^p g_{\boldsymbol{t}, j}'(W^j_{t_j})^{2}\mathrm{d} t_j. 
\end{align*} 
Note that, for $\boldsymbol{t}=(0, \ldots, 0, t_j, 0, \ldots, 0)$, we have $\aleph_{\boldsymbol{t}}=2t_j$. Therefore, for $\boldsymbol{t}=(t_1, \ldots, t_p)\in [0,\infty)^p$, we have 
\begin{align}\label{eqn::commontime_control_aux}
2\max_{1\le j\le p}t_j\le \aleph_{\boldsymbol{t}}\le 2\sum_{j=1}^p t_j. 
\end{align} 
There exists a time change $\boldsymbol{t}(t)=(t_1(t), \ldots, t_p(t))$ with 
\begin{equation}\label{eqn::time_change}
	\dot{t}_j(t)=g_{\boldsymbol{t}, j}'(W^j_{t_j})^{-2}, \qquad 1\le j\le p,
\end{equation}
such that $\boldsymbol{\eta}_t=\boldsymbol{\eta}_{\boldsymbol{t}(t)}$ is parameterized by common-time parameter. 
In this case, we have $\aleph_t=\aleph_{\boldsymbol{t}(t)}=2pt$. 

To show~\eqref{eqn::commontime_control}, we have the following observations.
\begin{itemize}
\item On the one hand, $g_{\boldsymbol{t}, j}'(W^j_{t_j})\in (0,1]$, thus $t_j'(t)\ge 1$ and $t_j(t)\ge t$.
\item On the other hand, combining~\eqref{eqn::commontime_control_aux} and $\aleph_t=\aleph_{\boldsymbol{t}(t)}=2pt$, we have 
\begin{align*}
\max_{1\le j\le p}t_j(t)\le pt\le \sum_{j=1}^p t_j(t). 
\end{align*}
This gives $t_j(t)\le pt$. 
\end{itemize}
Combining the two sides, we obtain~\eqref{eqn::commontime_control}. 
\end{proof}

Under common-time parameter, we denote the $p$-tuple by $\boldsymbol{\eta}_t=\boldsymbol{\eta}_{(t, \ldots, t)}$. 
We define the following normalized conformal transformations:
\begin{itemize}
\item $g_t^j$ is the conformal map from $\mathbb{H}\setminus\eta^j_{[0,t]}$ onto $\mathbb{H}$ with $\lim_{z\to\infty}|g_t^j(z)-z|=0$, $1\le j\le p$.
\item $g_t$ is the conformal map from $\mathbb{H}\setminus\cup_{j=1}^p\eta^j_{[0,t]}$ onto $\mathbb{H}$ with $\lim_{z\to\infty}|g_t(z)-z|=0$. Note that $g_t=g_{(t, \ldots, t)}$ and
\begin{align*}
g_t(z)=z+\frac{2pt}{z}+o(1/|z|), \qquad \text{as }z\to\infty.
\end{align*}
\item $g_{t,j}$ is the conformal map from $\mathbb{H}\setminus g^j_{t_j}\left(\cup_{i\neq j}\eta^i_{[0,t]}\right)$ onto $\mathbb{H}$ with $\lim_{z\to\infty}|g_{t,j}(z)-z|=0$, $1\le j\le p$. 
\end{itemize}
Denote by $W^j$ the driving function of $\eta^j$ and denote the driving function of the $p$-tuple $\boldsymbol{\eta}_t$, started from $X_0=(x_1, \ldots, x_p)$, by 
\begin{equation}\label{eqn::commontime_driving}
\boldsymbol{X}_t=(X^1_t, \ldots, X^p_t), \qquad\text{with }X^j_t=g_{t,j}(W^j_t), \qquad \text{for }1\le j\le p. 
\end{equation}
We denote by $(\mathcal{F}_t, t\ge 0)$ the filtration generated by $(\boldsymbol{\eta}_t, t\ge 0)$. We denote its lifetime by $T=\inf\{t: X_t^j=X_t^k \text{ for some }j\neq k\}$.

\subsection{Dyson Brownian motion}
\label{subsec::DysonBM}
In this section, we summarize conclusions for Dyson Brownian motion with parameter $\beta=8/\kappa$, i.e. the solution to the system of SDEs~\eqref{eqn::SDE_halfwatermelon_intro}:
\begin{equation*}
\mathrm{d} X^j_t=\sqrt{\kappa}\mathrm{d} B_t^j+\sum_{k\neq j}\frac{4}{X_t^j-X_t^k}\mathrm{d} t, \qquad 1\le j\le p, 
\end{equation*}
where $\{B^j\}_{1\le j\le p}$ are independent standard Brownian motions. 
When $\kappa <8$, we have $\beta >1$ and the lifetime of Dyson Brownian motion is $\infty$, i.e. solution to~\eqref{eqn::SDE_halfwatermelon_intro} exists for all time. The solution has quasi-invariant measure that we describe below. 
For each $v>0$ and $\boldsymbol{x}=(x_1,\ldots,x_{p})\in\mathcal{X}_p$, define
\begin{equation}\label{eqn::invariantdensity_def_general}
	f_v(\boldsymbol{x})=\prod_{1\leq j<k\leq p}(x_k-x_j)^v.
\end{equation}
Note that $\mathcal{Z}_{p}(\boldsymbol{x})$ in~\eqref{eqn::LZfusion_H} is exactly $f_{2/\kappa}(\boldsymbol{x})$. 

\begin{lemma}\label{lem::density*}
	Fix $\kappa\in(0,8)$. 
	Denote by $\mathsf{Q}_{p}(t;\boldsymbol{x},\cdot)$ the transition density for Dyson Brownian motion~\eqref{eqn::SDE_halfwatermelon_intro} with parameter $\beta=8/\kappa>1$. Then, for any $\boldsymbol{x},\boldsymbol{y}\in\mathcal{X}_{p}$,
	\begin{equation}\label{eqn::DysonBM_density_asy}
		\mathsf{Q}_{p}(t;\boldsymbol{x},\boldsymbol{y})=\mathcal{I}_{p}^{-1} \left(\frac{1}{\sqrt{\kappa t}}\right)^{\Lambda_p}f_{8/\kappa}(\boldsymbol{y})\exp\left(-\frac{ |\boldsymbol{y}|^2}{2\kappa t}\right)\left(1+O\left(\frac{|\boldsymbol{x}|}{\sqrt{t}}\right)\right),
	\end{equation}
	where 
	\begin{equation*}
		\Lambda_p=\frac{p(4p-4+\kappa)}{\kappa},
	\end{equation*}
	and $\mathcal{I}_{p}\in (0,\infty)$ is a normalization constant
	\begin{equation}\label{eqn::DysonBM_normalization}
		\mathcal{I}_{p}=\int_{\mathcal{X}_{p}}f_{8/\kappa}(\boldsymbol{x})e^{-\frac{1}{2}|\boldsymbol{x}|^2}\mathrm{d}\boldsymbol{x}.
	\end{equation}
\end{lemma}

\begin{proof}
	The explicit formula of transition density $\mathsf{Q}_{p}(t;\boldsymbol{x},\boldsymbol{y})$ can be obtained from~\cite[Equation~(4.5)]{RoslerGeneralizedHermiteDunklOperators}:
	\begin{equation*} 
		\mathsf{Q}_{p}(t;\boldsymbol{x},\boldsymbol{y})=\mathcal{I}_{p}^{-1}\left(\frac{1}{\sqrt{\kappa t}}\right)^{\Lambda_p}f_{8/\kappa}(\boldsymbol{y})\exp\left(-\frac{|\boldsymbol{x}|^2+|\boldsymbol{y}|^2}{2\kappa t}\right)J_{4/\kappa}^A\left(\frac{\boldsymbol{x}}{\sqrt{\kappa t}} ,\frac{\boldsymbol{y}}{\sqrt{\kappa t}}\right), \quad \boldsymbol{x},\boldsymbol{y}\in\mathcal{X}_p,
	\end{equation*}
	where $J_{4/\kappa}^A$ is multivariate Bessel function of type $A$ with multiplicity $4/\kappa$ and is analytic on $\mathbb{R}^{2N}\times \mathbb{R}^{2N}$ with $J_{4/\kappa}^A(\boldsymbol{0},\boldsymbol{y})=1$ (see~\cite[Section~2]{RoslerGeneralizedHermiteDunklOperators} or~\cite[Section~1]{VoitCLTforMultivariateBES}). Letting $|\boldsymbol{x}|/\sqrt{t}\to 0$, we obtain~\eqref{eqn::DysonBM_density_asy} and complete the proof.
\end{proof}

The following lemma will be useful in the proof of Proposition~\ref{prop::cvg_DysonBM}. 
\begin{lemma}\label{lem::finite_constant}
	Fix $\kappa\in (0,4], n\ge 1$ and $p=2n$ and $\alpha\in\mathrm{LP}_n$. The constant  
	\begin{equation}\label{eqn::ppf_normalization}
		\mathcal{J}_\alpha=\int_{\mathcal{X}_{2n}}f_{8/\kappa}(\boldsymbol{x})G_\alpha(\boldsymbol{x})^{-1}e^{-\frac{1}{2}|\boldsymbol{x}|^2}\mathrm{d}\boldsymbol{x}
	\end{equation}
	is finite, where $G_{\alpha}$ is the Green's function in~\eqref{eqn::Green}. 
\end{lemma}

\begin{proof}
	Note that 
	\begin{equation*}
		0<f_{8/\kappa}(\boldsymbol{x})G_\alpha(\boldsymbol{x})^{-1}e^{-\frac{1}{2}|\boldsymbol{x}|^2}
		=f_{8/\kappa}(\boldsymbol{x})\mathcal{Z}_\alpha(\boldsymbol{x})\mathcal{Z}_{2n}(\boldsymbol{x})^{-1}e^{-\frac{1}{2}|\boldsymbol{x}|^2}
		=f_{6/\kappa}(\boldsymbol{x})\mathcal{Z}_\alpha(\boldsymbol{x})e^{-\frac{1}{2}|\boldsymbol{x}|^2}.
	\end{equation*}
	From the power-law bound~\eqref{eqn::PLB_strong} and the relation~\eqref{eqn::parametersbc}, we know that $f_{6/\kappa}(\boldsymbol{x})\mathcal{Z}_\alpha(\boldsymbol{x})e^{-\frac{1}{2}|\boldsymbol{x}|^2}$ is integrable on $\mathcal{X}_{2n}$. Thus $\mathcal{J}_{\alpha}$ is a finite constant.
\end{proof}

Now, we rewrite the local martingale in Lemma~\ref{lem::halfwatermelon_mart} and explain its relation to Dyson Brownian motion. 
\begin{lemma}\label{lem::halfwatermelon_mart_common}
Fix $\kappa\in (0,4], p\ge 2$ and $\boldsymbol{x}=(x_1, \ldots, x_p)\in\mathcal{X}_p$. We use the parameters $\mathfrak{b}, \mathfrak{c}$ defined~\eqref{eqn::parametersbc}. 
We consider the relation between the following two measures.
\begin{itemize}
\item For each $j\in\{1, \ldots, p\}$, let $\eta^j$ be $\mathrm{SLE}_{\kappa}$ in $(\mathbb{H}; x_j,\infty)$ and let $\mathbb{P}_p$ be the probability measure on $\boldsymbol{\eta}=(\eta^1, \ldots, \eta^p)$ under which the curves are independent.
\item Recall from Definition~\ref{def::halfwatermelonSLE} that $\mathbb{Q}_{p}(\boldsymbol{x})$ denotes the law of half-$p$-watermelon $\mathrm{SLE}_{\kappa}$ in $(\mathbb{H}; \boldsymbol{x}, \infty)$. 
\end{itemize}
We parameterize $\boldsymbol{\eta}$ by common-time parameter $t$ and let $\boldsymbol{X}_{t}$ denote the driving function as in~\eqref{eqn::commontime_driving}. The Radon-Nikodym derivative between $\mathbb{Q}_{p}(\boldsymbol{x})$ and $\mathbb{P}_p$ when both measures are restricted to $\mathcal{F}_t$ and $\{T>t\}$ is given by 
\begin{equation}\label{eqn::RN_halfwatermelonvsindept}
\frac{\mathrm{d} \mathbb{Q}_{p}(\boldsymbol{x})[\cdot |_{\mathcal{F}_t}]}{\mathrm{d} \mathbb{P}_p[\cdot |_{\mathcal{F}_t\cap\{T>t\}}]}=\frac{M_t(\mathcal{Z}_{p})}{M_0(\mathcal{Z}_{p})}, 
\end{equation}
where $\mathcal{Z}_{p}$ is the partition function~\eqref{eqn::LZfusion_H} and 
\begin{equation}\label{eqn::halfwatermelon_mart_common}
M_t(\mathcal{Z}_{p})=\prod_{j=1}^p g_{t,j}'(W^j_t)^{\mathfrak{b}}\times \exp\left(\frac{\mathfrak{c}}{2}\mathfrak{m}_t\right)\times \mathcal{Z}_{p}(\boldsymbol{X}_t), 
\end{equation}
and $\mathfrak{m}_t$ is given by
\begin{align}\label{eqn::mt_blm_common}
	\mathfrak{m}_{t} = \mathfrak{m} \left(\mathbb{H};\eta_{[0,t]}^{1},\ldots,\eta_{[0,t]}^{p} \right).
\end{align}
Moreover, under $\mathbb{Q}_{p}(\boldsymbol{x})$, the driving function $\boldsymbol{X}_t$ satisfies the system of SDEs~\eqref{eqn::SDE_halfwatermelon_intro}. 
\end{lemma}
\begin{proof}
Eq.~\eqref{eqn::halfwatermelon_mart_common} follows from Lemma~\ref{lem::halfwatermelon_mart}. Combination of~\eqref{eqn::time_change} and~\eqref{eqn::halfwatermelon_mart_aux2} gives~\eqref{eqn::SDE_halfwatermelon_intro}. In the numerator of LHS of~\eqref{eqn::RN_halfwatermelonvsindept}, the driving function $\boldsymbol{X}_t$ under $\mathbb{Q}_{p}(\boldsymbol{x})$ has the same law as Dyson Brownian motion of parameter $\beta=8/\kappa\ge 2$ whose lifetime is $\infty$, thus we do not include $T>t$ in the numerator. 
\end{proof}

\subsection{Proof of Theorem~\ref{thm::transitiondensity}}
\label{subsec::transitiondensity_proof}
The proof of Theorem~\ref{thm::transitiondensity} relies on writing the multi-time local martingales for half-watermelon SLE and multichordal SLE in common-time parameter. The case for half-watermelon SLE is given in Lemma~\ref{lem::halfwatermelon_mart_common}. The case for multichordal SLE can be derived similarly. 

\begin{lemma}\label{lem::ppf_mart_common}
Fix $\kappa\in (0,4], n\ge 2$ and $\alpha\in\mathrm{LP}_n$. Fix $\boldsymbol{x}=(x_1, \ldots, x_{2n})\in\mathcal{X}_{2n}$. We use the parameters $\mathfrak{b}, \mathfrak{c}$ defined~\eqref{eqn::parametersbc}. 
We consider the relation between the following two measures.
\begin{itemize}
\item For each $j\in\{1, \ldots, 2n\}$, let $\eta^j$ be $\mathrm{SLE}_{\kappa}$ in $(\mathbb{H}; x_j,\infty)$ and let $\mathbb{P}_{2n}$ be the probability measure on $\boldsymbol{\eta}=(\eta^1, \ldots, \eta^{2n})$ under which the curves are independent.
\item Recall from Definition~\ref{def::globalnSLE} that $\mathbb{Q}_{\alpha}(\boldsymbol{x})$ denotes the law of multichordal $\mathrm{SLE}_{\kappa}$ in $(\mathbb{H}; \boldsymbol{x})$ associated to link pattern $\alpha$. 
\end{itemize}
We parameterize $\boldsymbol{\eta}$ by common-time parameter $t$ and let $\boldsymbol{X}_{t}$ denote the driving function as in~\eqref{eqn::commontime_driving} with $p=2n$. The Radon-Nikodym derivative between $\mathbb{Q}_{\alpha}(\boldsymbol{x})$ and $\mathbb{P}_{2n}$ when both measures are restricted to $\mathcal{F}_t$ and $T>t$ is given by 
\begin{equation}\label{eqn::RN_ppfvsindept}
\frac{\mathrm{d} \mathbb{Q}_{\alpha}(\boldsymbol{x})[\cdot|_{\mathcal{F}_t\cap\{T>t\}}]}{\mathrm{d}\mathbb{P}_{2n}[\cdot|_{\mathcal{F}_t\cap\{T>t\}}]}=\frac{M_t(\mathcal{Z}_{\alpha})}{M_0(\mathcal{Z}_{\alpha})}, 
\end{equation}
where $\mathcal{Z}_{\alpha}$ is pure partition function in Section~\ref{subsec::pre_NSLE} and 
\begin{equation}
M_t(\mathcal{Z}_{\alpha})=\prod_{j=1}^{2n}g_{t,j}'(W^j_t)^{\mathfrak{b}}\times \exp\left(\frac{\mathfrak{c}}{2}\mathfrak{m}_t\right)\times\mathcal{Z}_{\alpha}(\boldsymbol{X}_t),
\end{equation}
and $\mathfrak{m}_t$ is defined by
\[	\mathfrak{m}_{t} = \mathfrak{m} \left(\mathbb{H};\eta_{[0,t]}^{1},\ldots,\eta_{[0,t]}^{2n} \right).
\]
 Moreover, under $\mathbb{Q}_{\alpha}(\boldsymbol{x})$, the driving function $\boldsymbol{X}_t$ satisfies the system of SDEs~\eqref{eqn::SDE_ppf}. 
\end{lemma}
\begin{proof}
Eq.~\eqref{eqn::RN_ppfvsindept} follows from Lemma~\ref{lem::ppf_mart}. Combination of~\eqref{eqn::time_change} and~\eqref{eqn::ppf_mart_aux2} gives~\eqref{eqn::SDE_ppf}.
\end{proof}

Now, we are ready to complete the proof of Theorem~\ref{thm::transitiondensity}. 
\begin{proof}[Proof of Theorem~\ref{thm::transitiondensity}]
For each $t>0$, we obtain from Lemmas~\ref{lem::halfwatermelon_mart_common} and~\ref{lem::ppf_mart_common} the Radon-Nikodym derivative of $\mathbb{Q}_{\alpha}(\boldsymbol{x})$ against $\mathbb{Q}_{2n}(\boldsymbol{x})$ when both measures are restricted to $\mathcal{F}_t$ and $\{T>t\}$: 
\begin{align*}
\frac{\mathrm{d} \mathbb{Q}_{\alpha}(\boldsymbol{x})[\cdot |_{\mathcal{F}_t\cap\{T>t\}}]}{\mathrm{d} \mathbb{Q}_{2n}(\boldsymbol{x})[\cdot|_{\mathcal{F}_t}]}=\frac{M_t(\mathcal{Z}_{\alpha})/M_0(\mathcal{Z}_{\alpha})}{M_t(\mathcal{Z}_{2n})/M_0(\mathcal{Z}_{2n})}=\frac{\mathcal{Z}_{\alpha}(\boldsymbol{X}_t)/\mathcal{Z}_{\alpha}(\boldsymbol{x})}{\mathcal{Z}_{2n}(\boldsymbol{X}_t)/\mathcal{Z}_{2n}(\boldsymbol{x})}=\frac{G_{\alpha}(\boldsymbol{x})}{G_{\alpha}(\boldsymbol{X}_t)}. 
\end{align*}
The transition density for $\boldsymbol{X}_t$ under $\mathbb{Q}_{2n}(\boldsymbol{x})$ is $\mathsf{Q}_{2n}(t; \boldsymbol{x}, \boldsymbol{y})$, thus the above Radon-Nikodym derivative gives the transition density $\mathsf{Q}_{\alpha}(t; \boldsymbol{x}, \boldsymbol{y})$ in~\eqref{eqn::transitiondensity_ppf} as desired. 
\end{proof}

\subsection{Proof of Proposition~\ref{prop::cvg_DysonBM}}
\label{subsec::lifetime}

\begin{proof}[Proposition~\ref{prop::cvg_DysonBM}]
We first prove the estimate on the lifetime~\eqref{eqn::multichordal_lifetime}. 
We have
\begin{align*}
	\mathbb{Q}_{\alpha}(\boldsymbol{x})[T>t]=&\int_{\mathcal{X}_{2n}}\mathsf{Q}_{\alpha}(t; \boldsymbol{x}, \boldsymbol{y})\mathrm{d} \boldsymbol{y}\\
	=&\int_{\mathcal{X}_{2n}}\mathsf{Q}_{2n}(t; \boldsymbol{x}, \boldsymbol{y})\frac{G_{\alpha}(\boldsymbol{x})}{G_{\alpha}(\boldsymbol{y})}\mathrm{d} \boldsymbol{y}\tag{due to~\eqref{eqn::transitiondensity_ppf}}\\
	=&\int_{\mathcal{X}_{2n}}\mathcal{I}_{2n}^{-1} \left(\frac{1}{\sqrt{\kappa t}}\right)^{\Lambda_{2n}}f_{8/\kappa}(\boldsymbol{y})\exp\left(-\frac{|\boldsymbol{y}|^2}{2\kappa t}\right)\left(1+O\left(\frac{|\boldsymbol{x}|}{\sqrt{t}}\right)\right)\frac{G_{\alpha}(\boldsymbol{x})}{G_{\alpha}(\boldsymbol{y})}\mathrm{d} \boldsymbol{y}\tag{due to~\eqref{eqn::DysonBM_density_asy}}\\
	=&\mathcal{I}_{2n}^{-1} \left(\frac{1}{\sqrt{\kappa t}}\right)^{\Lambda_{2n}}G_{\alpha}(\boldsymbol{x})\left(1+O\left(\frac{|\boldsymbol{x}|}{\sqrt{t}}\right)\right)\int_{\mathcal{X}_{2n}}f_{8/\kappa}(\boldsymbol{y})\exp\left(-\frac{|\boldsymbol{y}|^2}{2\kappa t}\right)G_{\alpha}(\boldsymbol{y})^{-1}\mathrm{d} \boldsymbol{y}\\
	=&\mathcal{I}_{2n}^{-1} \left(\frac{1}{\sqrt{\kappa t}}\right)^{\Lambda_{2n}}G_{\alpha}(\boldsymbol{x})\left(1+O\left(\frac{|\boldsymbol{x}|}{\sqrt{t}}\right)\right)\mathcal{J}_{\alpha}\left(\frac{1}{\sqrt{\kappa t}}\right)^{-\Lambda_{2n}'}
\end{align*}
where the exponent $\Lambda_{2n}'$ is obtained by change of variables in integration: 
\[\Lambda_{2n}'=\frac{n(12n-12+3\kappa)}{\kappa}.\]
Note that 
\[\Lambda_{2n}-\Lambda_{2n}'=\frac{2n(8n-4+\kappa)}{\kappa}-\frac{n(12n-12+3\kappa)}{\kappa}=\frac{n(4n+4-\kappa)}{\kappa}=A_{2n}^+.\]
This completes the proof of~\eqref{eqn::multichordal_lifetime}.

Next, we prove the convergence in total variation distance~\eqref{eqn::cvg_tv}.  For any $t>s>0$, from~\eqref{eqn::multichordal_lifetime}, we have 
\begin{align*}
\frac{\mathrm{d} \mathbb{Q}_{\alpha}^{(t)}(\boldsymbol{x})[\cdot |_{\mathcal{F}_s}]}{\mathrm{d} \mathbb{Q}_{\alpha}(\boldsymbol{x})[\cdot |_{\mathcal{F}_s\cap\{T>s\}}]}=\frac{\mathbb{Q}_{\alpha}(\boldsymbol{X}_s)[T>t-s]}{\mathbb{Q}_{\alpha}(\boldsymbol{x})[T>t]}=\left(1-\frac{s}{t}\right)^{-\frac{1}{2}A_{2n}^+}\frac{G_{\alpha}(\boldsymbol{X}_s)}{G_{\alpha}(\boldsymbol{x})}\left(1+O\left(\frac{|\boldsymbol{X}_s|}{\sqrt{t}}\right)\right). 
\end{align*}
Combining with 
\begin{align*}
\frac{\mathrm{d} \mathbb{Q}_{\alpha}(\boldsymbol{x})[\cdot |_{\mathcal{F}_s\cap\{T>s\}}]}{\mathrm{d} \mathbb{Q}_{2n}(\boldsymbol{x})[\cdot|_{\mathcal{F}_s}]}=\frac{G_{\alpha}(\boldsymbol{x})}{G_{\alpha}(\boldsymbol{X}_s)}, 
\end{align*}
we conclude that 
\begin{align*}
\frac{\mathrm{d} \mathbb{Q}_{\alpha}^{(t)}(\boldsymbol{x})[\cdot |_{\mathcal{F}_s}]}{\mathrm{d} \mathbb{Q}_{2n}(\boldsymbol{x})[\cdot|_{\mathcal{F}_s}]}=\left(1-\frac{s}{t}\right)^{-\frac{1}{2}A_{2n}^+}\left(1+O\left(\frac{|\boldsymbol{X}_s|}{\sqrt{t}}\right)\right). 
\end{align*}
This implies 
\begin{align*}
\lim_{t\to\infty}\mathbb{Q}_{2n}(\boldsymbol{x})\left[\left(1-\frac{\mathrm{d} \mathbb{Q}_{\alpha}^{(t)}(\boldsymbol{x})[\cdot |_{\mathcal{F}_s}]}{\mathrm{d} \mathbb{Q}_{2n}(\boldsymbol{x})[\cdot|_{\mathcal{F}_s}]}\right)^+\right]=0, 
\end{align*}
which gives the desired convergence in total variation distance~\eqref{eqn::cvg_tv}. 
\end{proof}

\section{Gaussian free field level lines and Dyson Brownian motion $\beta=2$}
\label{sec::GFF_levellines}

We will give a very brief summary on Gaussian free field (GFF) and its level lines, more details can be found in~\cite{SheffieldGFFMath, SchrammSheffieldContinuumGFF, WangWuLevellinesGFFI}. 
Suppose $D\subset\mathbb{C}$ is a domain. The Sobolev space $H_0^1(D)$ is the Hilbert space closure of $C_0^{\infty}(D)$ with respect to the Dirichlet inner product
$(f,g)_{\nabla}=\frac{1}{2\pi}\int \nabla f\cdot \nabla g \; \mathrm{d} x$.
The zero-boundary GFF $\Gamma$ on $D$ is given by $\Gamma=\sum_n X_n f_n$ where $\{X_n\}$ is a sequence of i.i.d. normal Gaussian random variables and $\{f_n\}$ is an orthonormal basis for $H_0^1(D)$. Such sum does not converge in $H_0^1(D)$, but it does converge in an appropriate space of distributions. The GFF with boundary data $\Gamma_0$ is the sum of the zero-boundary GFF on $D$ and the function in $D$, still denoted by $\Gamma_0$, which is harmonic in $D$ and is equal to $\Gamma_0$ on $\partial D$. It is shown in~\cite{SchrammSheffieldContinuumGFF} that variants of $\mathrm{SLE}_4$ process can be coupled with GFF as ``level lines". In this section, we focus on the coupling between multichordal $\mathrm{SLE}_4$ and GFF. 

We fix $\lambda=\pi/2, n\ge 2$ and fix $\boldsymbol{x}=(x_1, \ldots, x_{2n})\in\mathcal{X}_{2n}$. We consider GFF $\Gamma$ in $\mathbb{H}$ with the following boundary data:
\begin{equation}\label{eqn::GFF_boundarydata}
	\lambda \text{ on } (x_{2j-1},x_{2j}) \text{ for } 1\leq j\leq n, \qquad -\lambda \text{ on } (x_{2j},x_{2j+1}) \text{ for } 0\leq j\leq n,
\end{equation}
with the convention that $x_0=-\infty$ and $x_{2n+1}=\infty$. For $1\le j\le n$, let $\eta^{2j-1}$ be the level line of $\Gamma$ starting from $x_{2j-1}$ and let $\eta^{2j}$ be the level line of $-\Gamma$ starting from $x_{2j}$. 
We denote the law of $\boldsymbol{\eta}=(\eta^1, \ldots, \eta^{2n})$ by $\mathbb{Q}_{\mathrm{GFF}}(\boldsymbol{x})$. 
We parameterize $\boldsymbol{\eta}$ by common-time parameter and denote its driving function by $\boldsymbol{X}_t=(X_t^1, \ldots, X_t^{2n})$. Denote by $(\mathcal{F}_t, t\ge 0)$ the filtration generated by $(\boldsymbol{\eta}_t, t\ge 0)$, and by $T=\inf\{t: X_t^j=X_t^k\text{ for some }j\neq k\}$ the lifetime.

We will prove in Lemma~\ref{lem::GFFlevellines_commontime_mart} that, under $\mathbb{Q}_{\mathrm{GFF}}(\boldsymbol{x})$, the driving function $\boldsymbol{X}_t$ satisfies the SDE
\begin{equation}\label{eqn::GFFlevellines_SDE}
	\mathrm{d} X_t^j=2\mathrm{d} B_t^j+\sum_{k\ne j}\frac{2\left((-1)^{j-k}+1\right)}{X_t^j-X_t^k}\mathrm{d} t, \qquad 1\leq j\leq 2n,
\end{equation}
where $\{B^j\}_{1\leq j\leq 2n}$ are independent standard Brownian motions. The main conclusion of this section is the following. 

\begin{proposition}\label{prop::GFF_DysonBM}
	Consider level lines $\boldsymbol{\eta}=(\eta^1, \ldots, \eta^{2n})$ of the GFF with boundary data~\eqref{eqn::GFF_boundarydata}. 
	\begin{itemize}
		\item Denote by $\mathsf{Q}_{\mathrm{GFF}}(t; \cdot, \cdot)$ the transition density for the solution to~\eqref{eqn::GFFlevellines_SDE}. 
		\item Denote by $\mathsf{Q}_{2n}(t; \cdot, \cdot)$ the transition density for Dyson Brownian motion with parameter $\beta=2$: 
		\begin{equation*}\label{eqn::DysonBM_kappa4}
			\mathrm{d} X_t^j=2\mathrm{d} B_t^j+\sum_{k\ne j}\frac{4}{X_t^j-X_t^k}\mathrm{d} t, \quad 1\leq j\leq 2n,
		\end{equation*}
		where $\{B^j\}_{1\leq j\leq 2n}$ are independent standard Brownian motions.
	\end{itemize}
	Then we have 
	\begin{equation}\label{eqn::transitiondensity_GFF}
		\mathsf{Q}_{\mathrm{GFF}}(t; \boldsymbol{x}, \boldsymbol{y})=\mathsf{Q}_{2n}(t; \boldsymbol{x}, \boldsymbol{y})\frac{G(\boldsymbol{x})}{G(\boldsymbol{y})},\qquad \text{for all }t\ge 0, \text{ and }\boldsymbol{x},\boldsymbol{y}\in\mathcal{X}_{2n},
	\end{equation}
	where 
	\begin{equation}\label{eqn::Greenfunction_GFF}
		G(x_1, \ldots, x_{2n})=\prod_{1\le j<k\le 2n}(x_k-x_j)^{\frac{1}{2}(1-(-1)^{j-k})}. 
	\end{equation}
	Moreover, 
	\begin{equation}\label{eqn::GFF_estimates}
		\mathbb{Q}_{\mathrm{GFF}}(\boldsymbol{x})[T>t]=\mathcal{I}_{2n}^{-1}\mathcal{J}_n G(\boldsymbol{x})\left(\frac{1}{2\sqrt{t}}\right)^{n^2}\left(1+O\left(\frac{|\boldsymbol{x}|}{\sqrt{t}}\right)\right), \quad \text{as } t\to\infty,
	\end{equation}
	where $\mathcal{I}_{2n}$ and $\mathcal{J}_n$ are constants: 
	\begin{equation*}
		\mathcal{I}_{2n}=(2\pi)^n\prod_{j=1}^{2n-1}j!,\qquad 
		\mathcal{J}_n=\int_{\mathcal{X}_{2n}}\prod_{1\leq j<k\leq 2n}(x_k-x_j)^{\frac{1}{2}(3+(-1)^{j-k})}\times e^{-\frac{1}{2}|\boldsymbol{x}|^2}\mathrm{d}\boldsymbol{x}.
	\end{equation*}
\end{proposition}

The proof for Proposition~\ref{prop::GFF_DysonBM} relies on Theorem~\ref{thm::transitiondensity} and the explicit partition function for GFF~\cite{PeltolaWuGlobalMultipleSLEs}: 
\begin{equation}\label{eqn::GFF_partitionfunction}
	\mathcal{Z}_{\mathrm{GFF}}^{(n)}(x_1, \ldots, x_{2n})=\prod_{1\le j<k\le 2n}(x_k-x_j)^{\frac{1}{2}(-1)^{j-k}}.
\end{equation}
Note that $\mathbb{Q}_{\mathrm{GFF}}(\boldsymbol{x})$ does not correspond to a single $\mathbb{Q}_{\alpha}(\boldsymbol{x})$ for any $\alpha\in\mathrm{LP}_n$. Instead, it is a linear combination of all of them. We consider pure partition functions $\{\mathcal{Z}_{\alpha}(\boldsymbol{x}), \alpha\in\mathrm{LP}_n\}$ and multichordal SLE measures $\{\mathbb{Q}_{\alpha}(\boldsymbol{x}): \alpha\in\mathrm{LP}_n\}$ in Section~\ref{subsec::pre_NSLE} with $\kappa=4$. 
From~\cite[Theorem~1.4]{PeltolaWuGlobalMultipleSLEs}, we have
	\begin{equation}\label{eqn::GFF_total_pure}
		\mathcal{Z}_{\mathrm{GFF}}^{(n)}(\boldsymbol{x})=\sum_{\alpha\in\mathrm{LP}_n}\mathcal{Z}_{\alpha}(\boldsymbol{x}),
	\end{equation}
	and
	\begin{equation}\label{eqn::connectprob}
		\mathbb{Q}_{\mathrm{GFF}}(\boldsymbol{x})=\sum_{\alpha\in\mathrm{LP}_n}\frac{\mathcal{Z}_{\alpha}(\boldsymbol{x})}{\mathcal{Z}_{\mathrm{GFF}}^{(n)}(\boldsymbol{x})}\mathbb{Q}_{\alpha}(\boldsymbol{x}).
	\end{equation}

\begin{lemma}\label{lem::GFFlevellines_commontime_mart}
Fix $\kappa=4, n\ge 1$ and $\boldsymbol{x}=(x_1, \ldots, x_{2n})\in\mathcal{X}_{2n}$. We consider the relation between the following two measures.
\begin{itemize}
\item For each $j\in\{1, \ldots, 2n\}$, let $\eta^j$ be $\mathrm{SLE}_{4}$ in $(\mathbb{H}; x_j,\infty)$ and let $\mathbb{P}_{2n}$ be the probability measure on $\boldsymbol{\eta}=(\eta^1, \ldots, \eta^{2n})$ under which the curves are independent.
\item Denote by $\mathbb{Q}_{\mathrm{GFF}}({\boldsymbol{x}})$ the law of level lines $\boldsymbol{\eta}=(\eta^1, \ldots, \eta^{2n})$ of GFF $\Gamma$ with boundary data~\eqref{eqn::GFF_boundarydata}.  
\end{itemize}
We parameterize $\boldsymbol{\eta}$ by common-time parameter $t$ by let $\boldsymbol{X}_t$ denote the driving function as in~\eqref{eqn::commontime_driving} with $p=2n$. The Radon-Nikodym derivative between $\mathbb{Q}_{\mathrm{GFF}}(\boldsymbol{x})$ and $\mathbb{P}_{2n}$ when both measures are restricted to $\mathcal{F}_t$ and $\{T>t\}$ is given by 
\begin{equation}
\frac{\mathrm{d} \mathbb{Q}_{\mathrm{GFF}}(\boldsymbol{x})[\cdot|_{\mathcal{F}_t\cap\{T>t\}}]}{\mathrm{d} \mathbb{P}_{2n}[\cdot|_{\mathcal{F}_t\cap\{T>t\}}]}=\frac{M_t(\mathcal{Z}_{\mathrm{GFF}}^{(n)})}{M_0(\mathcal{Z}_{\mathrm{GFF}}^{(n)})}, 
\end{equation}
where $\mathcal{Z}_{\mathrm{GFF}}^{(n)}$ is the partition function~\eqref{eqn::GFF_partitionfunction} and 
\begin{equation}
M_t(\mathcal{Z}_{\mathrm{GFF}}^{(n)})=\prod_{j=1}^{2n}g_{t,j}'(W^j_t)^{1/4}\times \exp\left(\frac{1}{2}\mathfrak{m}_t\right)\times\mathcal{Z}_{\mathrm{GFF}}^{(n)}(\boldsymbol{X}_t),
\end{equation}
and $\mathfrak{m}_t$ is given by\[	\mathfrak{m}_{t} = \mathfrak{m} \left(\mathbb{H};\eta_{[0,t]}^{1},\ldots,\eta_{[0,t]}^{2n} \right).
\]
 Moreover, under $\mathbb{Q}_{\mathrm{GFF}}(\boldsymbol{x})$, the driving function $\boldsymbol{X}_t$ satisfies the system of SDEs~\eqref{eqn::GFFlevellines_SDE}. 
\end{lemma}

\begin{proof}
From~\eqref{eqn::GFF_total_pure} and~\eqref{eqn::connectprob}, we have
	\begin{align*}
		\frac{\mathrm{d}\mathbb{Q}_{\mathrm{GFF}}(\boldsymbol{x})[\cdot|_{\mathcal{F}_t\cap\{T>t\}}]}{\mathrm{d}\mathbb{P}_{2n}[\cdot|_{\mathcal{F}_t\cap\{T>t\}}]}
		=&\sum_{\alpha\in\mathrm{LP}_n}\frac{\mathcal{Z}_{\alpha}(\boldsymbol{x})}{\mathcal{Z}_{\mathrm{GFF}}^{(n)}(\boldsymbol{x})}\times
		\frac{\mathrm{d}\mathbb{Q}_{\alpha}(\boldsymbol{x})[\cdot|_{\mathcal{F}_t\cap\{T>t\}}]}{\mathrm{d}\mathbb{P}_{2n}[\cdot|_{\mathcal{F}_t\cap\{T>t\}}]}\tag{due to~\eqref{eqn::connectprob}}\\
		=&\sum_{\alpha\in\mathrm{LP}_n}\frac{\mathcal{Z}_{\alpha}(\boldsymbol{x})}{\mathcal{Z}_{\mathrm{GFF}}^{(n)}(\boldsymbol{x})}\times\frac{M_{t}(\mathcal{Z}_{\alpha})}{M_0(\mathcal{Z}_{\alpha})}\tag{due to Lemma~\ref{lem::ppf_mart_common} with $\kappa=4$}\\
		=&\frac{M_t(\mathcal{Z}_{\mathrm{GFF}}^{(n)})}{M_0(\mathcal{Z}_{\mathrm{GFF}}^{(n)})}. \tag{due to~\eqref{eqn::GFF_total_pure}}
	\end{align*} 
	Note that $\mathcal{Z}_{\mathrm{GFF}}^{(n)}$ is a combination of $\mathcal{Z}_{\alpha} $ as in~\eqref{eqn::GFF_total_pure} and hence it also satisfies the system of PDEs~\eqref{eqn::PDE} with $\kappa=4$. Then similar calculation as in the proof of Lemma~\ref{lem::ppf_mart_common} implies that, under $\mathbb{Q}_{\mathrm{GFF}}(\boldsymbol{x})$, the driving function $\boldsymbol{X}_t$ satisfies the system of SDEs: 
	\begin{equation*}
		\mathrm{d} X_t^j=2\mathrm{d} B_t^j+4(\partial_j\log\mathcal{Z}_{\mathrm{GFF}}^{(n)})(X_t^1,\ldots,X_t^{2n})\mathrm{d} t+\sum_{k\ne j}\frac{2}{X_t^j-X_t^k}\mathrm{d} t, \quad 1\leq j\leq 2n,
	\end{equation*}
	where $\{B^j\}_{1\leq j\leq 2n}$ are independent standard Brownian motions. From the explicit expression~\eqref{eqn::GFF_partitionfunction}, this is the same as SDE~\eqref{eqn::GFFlevellines_SDE} and we complete the proof.
\end{proof}

\begin{proof}[Proof of Proposition~\ref{prop::GFF_DysonBM}]
	We fix $\kappa=4$. From Lemma~\ref{lem::halfwatermelon_mart_common} with $p=2n$ and $\kappa=4$ and Lemma~\ref{lem::GFFlevellines_commontime_mart}, the Radon-Nikodym derivative of $\mathbb{Q}_{\mathrm{GFF}}(\boldsymbol{x})$ against $\mathbb{Q}_{2n}(\boldsymbol{x})$ when both measures are restricted to $\mathcal{F}_t$ and $\{T>t\}$ is given by : 
	\begin{equation*}
		\frac{\mathrm{d}\mathbb{Q}_{\mathrm{GFF}}(\boldsymbol{x})[\cdot|_{\mathcal{F}_t\cap\{T>t\}}]}{\mathrm{d}\mathbb{Q}_{2n}(\boldsymbol{x})[\cdot|_{\mathcal{F}_t}]}
		=\frac{M_t(\mathcal{Z}_{\mathrm{GFF}}^{(n)})/M_0(\mathcal{Z}_{\mathrm{GFF}}^{(n)})}{M_t(\mathcal{Z}_{2n})/M_0(\mathcal{Z}_{2n})}
		=\frac{\mathcal{Z}_{\mathrm{GFF}}^{(n)}(\boldsymbol{X}_t)/\mathcal{Z}_{\mathrm{GFF}}^{(n)}(\boldsymbol{x})}{\mathcal{Z}_{2n}(\boldsymbol{X}_t)/\mathcal{Z}_{2n}(\boldsymbol{x})}=\frac{G(\boldsymbol{x})}{G(\boldsymbol{X}_t)},
	\end{equation*}
	where
	\begin{equation*}
		G(\boldsymbol{x})=\mathcal{Z}_{2n}(\boldsymbol{x})/\mathcal{Z}_{\mathrm{GFF}}^{(n)}(\boldsymbol{x})=\prod_{1\leq j<k\leq 2n}(x_k-x_j)^{\frac{1}{2}(1-(-1)^{j-k})}.
	\end{equation*}
	Since $\boldsymbol{X}_t$ is a Markov process with transition density $\mathsf{Q}_{2n}(t;\boldsymbol{x},\boldsymbol{y})$ under $\mathbb{Q}_{2n}(\boldsymbol{x})$, the Radon-Nikodym derivative gives the transition density $\mathsf{Q}_{\mathrm{GFF}}(t;\boldsymbol{x},\boldsymbol{y})$ as desired in~\eqref{eqn::transitiondensity_GFF}. 
	
	Lemma~\ref{lem::density*} gives the asymptotic of transition density for Dyson Brownian motion with parameter $\beta=2$:
	\begin{align}\label{eqn::DysonBM_asy_beta2}
	\mathsf{Q}_{2n}(t;\boldsymbol{x},\boldsymbol{y})=\mathcal{I}_{2n}^{-1}(4t)^{-2n^2}f_2(\boldsymbol{y})\exp\left(-\frac{|\boldsymbol{y}|^2}{8t}\right)\left(1+O\left(\frac{|\boldsymbol{x}|}{\sqrt{t}}\right)\right), \quad \text{as } t\to\infty,
	\end{align}
	where (see~\cite[Eq.~(17.6.7)]{MehtaRandomMatrices})
	\begin{align*}
		\mathcal{I}_{2n}=\int_{\mathcal{X}_{2n}}f_{2}(\boldsymbol{x})e^{-\frac{1}{2}|\boldsymbol{x}|^2}\mathrm{d}\boldsymbol{x}=\frac{1}{(2n)!}\int_{\mathbb{R}^{2n}}|f_{2}(\boldsymbol{x})|e^{-\frac{1}{2}|\boldsymbol{x}|^2}\mathrm{d}\boldsymbol{x}=(2\pi)^n\prod_{j=1}^{2n-1}j!. 
	\end{align*}
	Then we obtain 
	\begin{align*}
		\mathbb{Q}_{\mathrm{GFF}}(\boldsymbol{x})[T>t]
		&=\int_{\mathcal{X}_{2n}}\mathsf{Q}_{\mathrm{GFF}}(t;\boldsymbol{x},\boldsymbol{y})\mathrm{d}\boldsymbol{y}\\=&\int_{\mathcal{X}_{2n}}\mathsf{Q}_{2n}(t;\boldsymbol{x},\boldsymbol{y})\frac{G(\boldsymbol{x})}{G(\boldsymbol{y})}\mathrm{d}\boldsymbol{y}\tag{due to~\eqref{eqn::transitiondensity_GFF}}\\
		&=\mathcal{I}_{2n}^{-1}G(\boldsymbol{x})(4t)^{-2n^2}\left(1+O\left(\frac{|\boldsymbol{x}|}{\sqrt{t}}\right)\right)\int_{\mathcal{X}_{2n}}\frac{f_2(\boldsymbol{y})}{G(\boldsymbol{y})}\exp\left(-\frac{|\boldsymbol{y}|^2}{8t}\right)\mathrm{d}\boldsymbol{y}\tag{due to~\eqref{eqn::DysonBM_asy_beta2}}\\
		&=\mathcal{I}_{2n}^{-1}G(\boldsymbol{x})(4t)^{-2n^2}\left(1+O\left(\frac{|\boldsymbol{x}|}{\sqrt{t}}\right)\right)  \mathcal{J}_n (4t)^{3n^2/2} \\
		&=\mathcal{I}_{2n}^{-1}\mathcal{J}_n G(\boldsymbol{x})(4t)^{-n^2/2}\left(1+O\left(\frac{|\boldsymbol{x}|}{\sqrt{t}}\right)\right), \quad \text{as } t\to\infty, 
	\end{align*}
	where
	\begin{align*}
		\mathcal{J}_n=\int_{\mathcal{X}_{2n}}\frac{f_2(\boldsymbol{x})}{G(\boldsymbol{x})}e^{-\frac{1}{2}|\boldsymbol{x}|^2}\mathrm{d}\boldsymbol{x}=\int_{\mathcal{X}_{2n}}\prod_{1\leq j<k\leq 2N}(x_k-x_j)^{\frac{1}{2}(3+(-1)^{j-k})}e^{-\frac{1}{2}|\boldsymbol{x}|^2}\mathrm{d}\boldsymbol{x}. 
	\end{align*}	
	This completes the proof of~\eqref{eqn::GFF_estimates}. 
	\end{proof}
\section{Proof of Theorem~\ref{thm::multichordalSLE_boundaryarm}}
\label{sec::NSLE_estimate}

The size ``$\mathrm{hsiz}$" is in general defined for compact hulls. Suppose $K$ is an $\mathbb{H}$-hull and define $\mathrm{hsiz}(K)$ to be the two-dimensional Lebesgue measure of the union of all balls centered by points in $K$ that are tangent to the real line, i.e.
\begin{equation*}
\mathrm{hsiz}(K)=\mathrm{area}\left(\cup_{u+\mathfrak{i} v\in K} B_v(u+\mathfrak{i} v)\right). 
\end{equation*} 
For a curve $\gamma\in \mathfrak{X}(\mathbb{H}; x, y)$, the domain surrounded by $\gamma$ and the real line is an $\mathbb{H}$-hull, then $\mathrm{hsiz}$ of this compact hull is the same as we defined in~\eqref{eqn::hsiz_def} in the introduction. The size $\mathrm{hsiz}$ is comparable to half-plane capacity\footnote{We define $\mathrm{hcap}(K)=\lim_{z\to\infty} z (g_K(z)-z)/2$ while the half plane capacity was defined as $\mathrm{hcap}(K)=\lim_{z\to\infty} z (g_K(z)-z)$ in~\cite{LLN09}.}~\cite[Theorem~1]{LLN09}: 
\begin{equation}\label{eqn::hsiz_hcap}
\frac{1}{132}\mathrm{hsiz}(K)< \mathrm{hcap}(K)<\frac{7}{4\pi}\mathrm{hsiz}(K). 
\end{equation}

\begin{proof}[Proof of Theorem~\ref{thm::multichordalSLE_boundaryarm}]
We have
\begin{align*}
	\mathbb{Q}_{\alpha}(\boldsymbol{x})\left[\mathrm{hsiz}(\gamma^j)>R^2,\forall 1\le j\le n\right] \le &\mathbb{Q}_{\alpha}(\boldsymbol{x})\left[\mathrm{hcap}(\gamma^j)>R^2/132,\forall 1\le j\le n\right] \tag{due to~\eqref{eqn::hsiz_hcap}}\\
	\le & \mathbb{Q}_{\alpha}(\boldsymbol{x})\left[T>R^2/(528n)\right],\tag{due to~\eqref{eqn::commontime_control}}\\
	\mathbb{Q}_{\alpha}(\boldsymbol{x})\left[\mathrm{hsiz}(\gamma^j)>R^2,\forall 1\le j\le n\right] \ge &\mathbb{Q}_{\alpha}(\boldsymbol{x})\left[\mathrm{hcap}(\gamma^j)> 7R^2/(4\pi),\forall 1\le j\le n\right] \tag{due to~\eqref{eqn::hsiz_hcap}}\\
	\ge &\mathbb{Q}_{\alpha}(\boldsymbol{x})\left[T> 7R^2/(4\pi)\right]. \tag{due to~\eqref{eqn::commontime_control}}
\end{align*}
Combining with~\eqref{eqn::multichordal_lifetime}, we obtain the desired result.
\end{proof}

Let us discuss the relation between Theorem~\ref{thm::multichordalSLE_boundaryarm} and~\cite{LawlerMinkowskiSLERealLine, ZhanGreen2SLEboundary}. For $R>0$, denote by $B_R$ the ball centered at origin with radius $R$. We expect the following conclusion holds:\footnote{For two functions $F_1(\boldsymbol{x}; R)$ and $F_2(\boldsymbol{x}; R)$, we write $F_1(\boldsymbol{x}; R)\sim F_2(\boldsymbol{x}; R)$ as $R\to \infty$ if the limit $\lim_{R\to\infty}\frac{F_1(\boldsymbol{x}; R)}{F_2(\boldsymbol{x}; R)}$ exists and takes value in $(0,\infty)$. }
\begin{equation}\label{eqn::Greenfunction_asymp_stronger}
\mathbb{Q}_{\alpha}(\boldsymbol{x})\left[\gamma^j\cap\partial B_R\neq \emptyset, \forall 1\le j\le n\right]\sim G_{\alpha}(\boldsymbol{x})R^{-A_{2n}^+}, \qquad \text{as }R\to\infty
\end{equation}
This stronger estimate~\eqref{eqn::Greenfunction_asymp_stronger} is proved for $\kappa\in(0,8), n=1$ in~\cite[Theorem~1]{LawlerMinkowskiSLERealLine} and is proved for $\kappa\in(0,8), n=2$ in~\cite[Theorem~1.1]{ZhanGreen2SLEboundary}. 
For the general case $n\ge 3$, it is easy to get the estimate on one-side: 
\begin{align}\label{eqn::multichordal_ball}
\mathbb{Q}_{\alpha}(\boldsymbol{x})\left[\gamma^j\cap\partial B_R\neq \emptyset, \forall 1\le j\le n\right]\gtrsim G_{\alpha}(\boldsymbol{x})R^{-A_{2n}^+}; 
\end{align}
because $\mathrm{hcap}(\gamma^j)> R^2/2$ implies $\gamma^j\cap\partial B_R\neq \emptyset$. However, there is still gap between~\eqref{eqn::multichordal_ball} and~\eqref{eqn::Greenfunction_asymp_stronger} that we do not know how to address. %

{\small

	}
\end{document}